\documentclass[a4paper,12pt,reqno]{amsart}
\usepackage{comment} 
\usepackage[utf8]{inputenc} 
\usepackage{fullpage}
\usepackage{graphicx}
\usepackage{hyperref}
\usepackage[lite]{amsrefs}
\usepackage{titlesec}
\usepackage{amsmath,amssymb,amsthm}
\usepackage{tikz}
\usetikzlibrary{decorations.markings}
\usetikzlibrary{decorations}

\newtheorem{theorem}{Theorem}[section]

\newtheorem{proposition}[theorem]{Proposition}
\newtheorem{lemma}[theorem]{Lemma}

\newtheorem{corollary}[theorem]{Corollary}

\theoremstyle{definition}

\newtheorem{remark}[theorem]{Remark}
\newtheorem{definition}[theorem]{Definition}

\numberwithin{equation}{section}

\makeatletter
\def\@secnumfont{\bfseries}
\def\section{\@startsection{section}{1}%
	\z@{.7\linespacing\@plus\linespacing}{.5\linespacing}%
	{\normalfont\fontsize{12}{17}\sffamily\bfseries}}
\def\subsection{\@startsection{subsection}{2}%
	\z@{.5\linespacing\@plus.7\linespacing}{-.5em}%
	{\normalfont\fontsize{10}{14}\sffamily\bfseries}}
\makeatother

\titleformat{\section}
{\normalfont\fontsize{12}{17}\sffamily\bfseries}
{\thesection}
{1em}
{}

\titleformat{\subsection}
{\normalfont\fontsize{10}{14}\sffamily\bfseries}
{\thesubsection}
{1em}
{}

\titleformat{\subsubsection}
{\normalfont\fontsize{10}{14}\sffamily}
{\thesubsubsection}
{1em}
{}

\newcommand{\norm}[1]{\lVert #1 \rVert}
\newcommand{\abs}[1]{| #1 |}
\newcommand{\scal}[1]{\left\langle  #1 \right\rangle}

\newcommand{\Rr}{\mathbb{R}}

\newcommand{\Ss}{\mathbb{S}}
\newcommand{\Cc}{\mathbb{C}}
\newcommand{\Hh}{\mathbb{H}}
\newcommand{\Zz}{\mathbb{Z}}

\newcommand{\SU}{\mathrm{SU}}
\renewcommand{\Re}{\mathop{\mathrm{Re}}}
\renewcommand{\Im}{\mathop{\mathrm{Im}}}

\newcommand{\la}{\lambda}
\newcommand{\x}{\mathbf{x}}
\newcommand{\n}{\mathbf{n}}

\newcommand{\hfrak}{\mathfrak{h}}
\newcommand{\vfrak}{\mathfrak{v}}
\newcommand{\inv}{^{-1}}
\newcommand{\tr}{\mathrm{tr}}
\newcommand{\ubar}[1]{\underline{#1}}

\renewcommand{\and}{\quad \text{and} \quad}

\renewcommand{\matrix}[4]{
	\begin{pmatrix}
		#1 & #2 \\ #3 & #4
	\end{pmatrix}
}
\renewcommand{\vector}[2]{
	\begin{pmatrix}
		#1 \\ #2 
	\end{pmatrix}
}

\title[Constant curvature rotational nets and periodic B\"acklund transforms]
{Constant curvature rotational nets and periodic B\"acklund transforms}

\author[]{T. Raujouan}
\author[]{W. Rossman}
\author[]{N. Suda}

\begin{document}
	
\begin{abstract}
	After giving explicit parametrizations of 
	discrete constant negative Gaussian curvature surfaces (negative CGC, i.e. 
	discrete pseudospherical surfaces) of revolution,
	we construct B\"acklund transformations that again will have explicit parametrizations and 
	are new examples of non-rotational discrete pseudospherical surfaces. 
  In the process of doing this, for discrete CGC circular nets, 
  we can provide rotationally invariant families of flat connections and give conditions on them
  so that the B\"acklund transformations preserve periodicity, that is, have annular topology.
\end{abstract}

\maketitle

\section*{Introduction}
Bobenko and Pinkall \cite{BP} investigated Lax pairs
corresponding to the discrete constant 
negative Gaussian curvature surfaces introduced by Wunderlich \cite{WW} and 
Sauer \cite{SR}, establishing a framework based on the 
discretization of asymptotic coordinates. 
Furthermore, Bobenko and Pinkall \cite{BP3}, as well as Pedit and Wu \cite{PW}, investigated 
Lax pairs for discrete
constant positive Gaussian curvature surfaces and their parallel surfaces, 
which includes discrete constant mean curvature surfaces. 
In addition, Bobenko and Pinkall \cite{BP2} also addressed isothermic nets as a broader class which contains discrete constant mean curvature surfaces, 
where the isothermic nets — discretizations of conformal curvature line coordinates — are defined in terms of cross-ratios.
Bobenko, Pottmann, and Wallner \cite{BPW} demonstrated that 
for circular nets, which constitute a discretization of curvature line coordinates, 
the principal curvatures, Gaussian curvature, and mean curvature can be defined via discrete analog of the Steiner formula.
Subsequently, Hoffmann and Sageman-Furnas \cite{HS} constructed a Lax pair for circular nets of constant negative Gaussian curvature (hereafter abbreviated as cK-nets),
building on the Lax pair introduced in \cite{BP}.
Furthermore, they computed the Gaussian curvature of the associated family arising from the Lax pair with the spectral parameter $\la=e^t$, 
employing edge-constraint nets and their curvature as introduced by Hoffmann, Sageman-Furnas, and Wardetzky \cite{HSW}.
Hoffmann and Sageman-Furnas \cite{HS} elucidated the relationship between asymptotic coordinates and curvature line coordinates for discrete surfaces, and further revisited the Bäcklund transformation for cK-nets introduced by Schief \cite{S} 
as gauge transformations of Lax pairs, and the Bäcklund transformation of \cite{S} is interpreted in \cite{HS} as a double Bäcklund transformation. 
Moreover, explicit parametrizations of the Kuen surface and the breather surface as circular nets are presented in \cite{HS}.

On the other hand, the third author \cite{SN} provided parametrizations of rotationally symmetric circular nets (hereafter abbreviated as rc-nets) with constant Gaussian curvature
derived from the definition of the Gaussian curvature, using the Steiner formula. 
In the case of positive constant Gaussian curvature, this is highly related to work of Bobenko and Hoffmann, because such surfaces are parallel to discrete constant mean curvature 
surfaces of revolution as studied in \cite{BH}, where they employed a unified approach which includes the notion of s-isothermic found in \cite{BP3,BHS}.
These parametrizations of constant Gaussian curvature rc-nets have made the following questions more accessible:
\begin{itemize}
	\item For smooth surfaces, it is well known that obtaining a surface of revolution 
	with constant Gaussian curvature leads to a second-order linear differential equation 
	whose solutions are written by trigonometric, hyperbolic or exponential functions. 
	Does an analogous difference equation arise in the case of circular nets?
	\item How is the rotational symmetry reflected in the Lax pair associated with rc-nets of constant Gaussian curvature?
\end{itemize}
Periodicity has been well studied in this context.
Periodic breather surfaces can be obtained from 
the straight line by applying a double Bäcklund transformation \cite{BL2, RS}. 
In the case of constant mean curvature surfaces, periodic examples such 
as bubbletons and Delaunay bubbletons have been extensively studied 
through Darboux transformations \cite{D, BL,BL3, SW, CLO}. 
In addition, doubly periodic tori with constant curvature have also been 
investigated from the viewpoint of integrable systems \cite{SJ, WR, WH, AU, PS, B1,B2, CS, MS, UT}. 
Thus, whether or not a surface exhibits periodicity constitutes an 
important theme in the study of surfaces with integrable systems. 
This leads us to the following query:
\begin{itemize}
	\item For rc-nets with constant negative Gaussian curvature, under what conditions do the cK-nets 
	obtained via the Bäcklund transformation defined in \cite{HS} preserve periodicity, and thus remain annular?
\end{itemize}
This paper concentrates mainly on the above three questions.

More specifically, in Section 1 we reconsider the
Lax pair for circular nets of constant Gaussian curvature, 
interpreting it as a flat connection within the framework of 
discrete gauge theory \cite{B, BCHPR, BHRS, NORYPJ}. In Section 2, we first demonstrate that, 
for rc-nets with constant Gaussian curvature, trigonometric, hyperbolic, and exponential 
functions arise as solutions to the associated difference equations obtained from the definition of curvature. 
As one consequence, we clarify how the singular vertices defined by Yasumoto and the second author \cite{RY} appear in circular nets of constant 
Gaussian curvature while employing Jacobi elliptic functions. In Section 3, 
for rc-nets, we present several results on flat connections that are invariant in one of the two lattice directions, which leads to their explicit description. 
In Section 4, after 
showing that the transformation of a certain related unitary function 
can be expressed in terms of a rational difference equation, we establish results on the periodicity 
of rc-nets with constant negative Gaussian curvature and their single and double Bäcklund transformations. 
Finally, we also discuss a method of linearization for the rational difference equations associated with 
rc-nets.

\begin{figure}[htbp]
	\setlength{\tabcolsep}{20pt}
	\begin{tabular}{cc}
	\includegraphics[width=5.0cm, trim=4.5cm 1cm 4.5cm 1cm, clip]{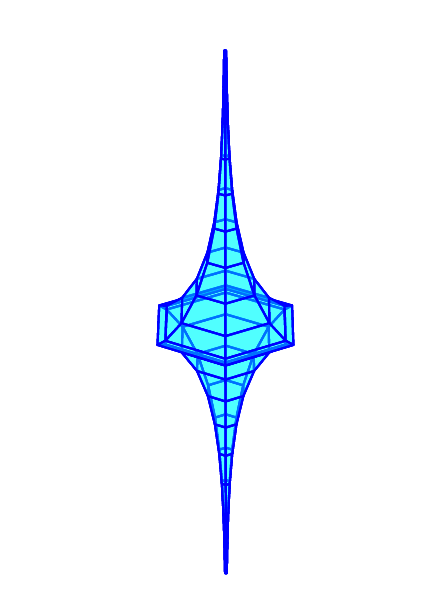}&
	\includegraphics[width=5.0cm, trim=4.5cm 1cm 4.5cm 1cm, clip]{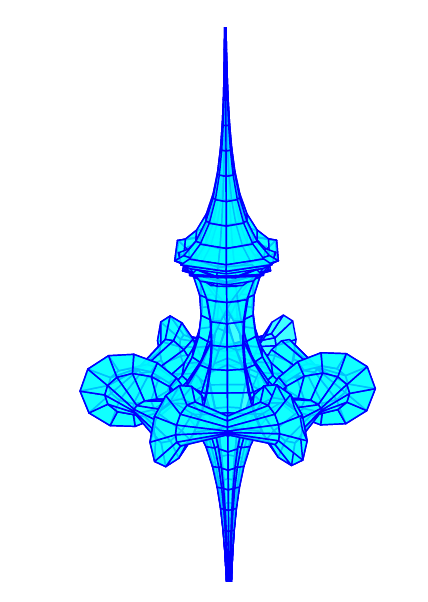}
	\end{tabular}
	\caption{Examples of periodic cK-nets (which are seen again in Figures \ref{figure:singular} and \ref{figure:backlund}).}
\end{figure}

\noindent\textbf{Acknowledgements.} The authors thank Alexander I. Bobenko, Francis Burstall, 
Joseph Cho, Masaya Hara, Christian M\"uller, Andrew Sageman-Furnas and Masashi Yasumoto for 
helpful advice and discussions. 
The authors gratefully acknowledge, respectively, support from
1) JSPS under the FY2022 JSPS Postdoctoral Fellowship P22766, 
2) the Japanese government Kiban C grant 23K03091,
3) JST SPRING, Grant Number JPMJSP2148.

\section{Surfaces from integrable systems}

\subsection{Quaternionic model of $\Rr^3$}

We denote by $\Hh$ the field of quaternions, which we realize as the real 4-dimensional vector space generated by the matrices $(\boldsymbol{\sigma}_0, -i\boldsymbol{\sigma}_1, -i\boldsymbol{\sigma}_2, -i\boldsymbol{\sigma}_3)$ where
\begin{equation}
	\boldsymbol{\sigma}_0 = \matrix{1}{0}{0}{1},\quad \boldsymbol{\sigma}_1 = \matrix{0}{1}{1}{0},\quad \boldsymbol{\sigma}_2 = \matrix{0}{-i}{i}{0},\quad \boldsymbol{\sigma}_3 = \matrix{1}{0}{0}{-1}.
\end{equation}
We identify the Euclidean space $\Rr^3$ with $\Im \Hh$ endowed with the norm $\norm{x}^2 = \det x$ and the orientation that agrees with $x \times y = \frac{1}{2}[x,y]$. Note that for all $x,y\in\Im \Hh$,
\begin{equation}
	x\cdot y = \frac{-\tr(xy)}{2},\quad x\times y = (xy)^{\tr=0}.
\end{equation}
where $(x)^{\tr =0} := x - \frac{\tr x}{2}\boldsymbol{\sigma}_0$ for $x\in \Hh$ is the projection onto $\Im \Hh$.

\subsection{Discrete gauge theory}

Let $\Gamma = (V,E, F)$ be a quad-graph with vertices $V$, oriented edges $E$, and quadrilateral faces $F$. 
We consider the set of all oriented edges $E$ as the union of a set $E^+$ of oriented edges in a chosen direction and 
the set $E^{-}$ of oppositely oriented edges.

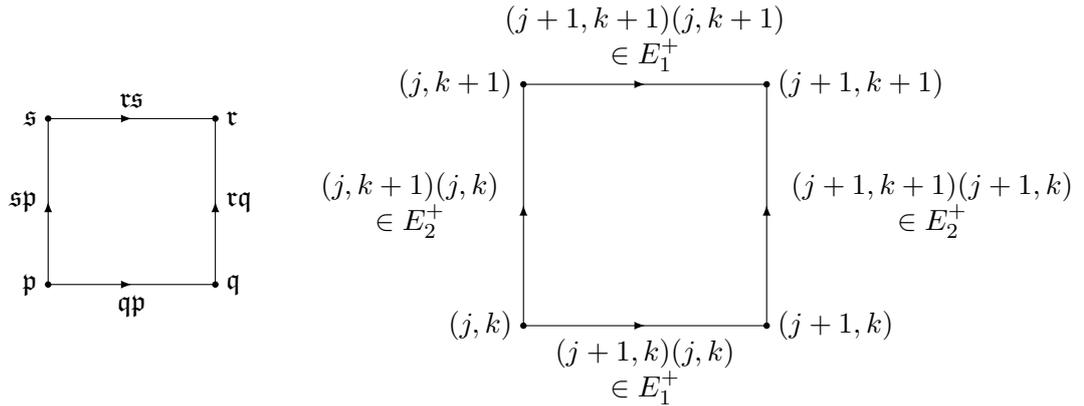
\begin{figure}[htbp]
  \centering
  \begin{minipage}{0.5\textwidth}
    \centering
		\hspace*{-9.0cm}
    \begin{tikzpicture}[
      x=1.1cm, y=1.1cm,
      every node/.style={font=\small},
      decoration={markings, mark=at position 0.5 with {\arrow{latex}}}
      ]
      \draw[black,postaction={decorate}] (0,0) -- (2,0)
        node[midway, below] {$\mathfrak{qp}$};
      \draw[black,postaction={decorate}] (2,0) -- (2,2)
        node[midway, right] {$\mathfrak{rq}$};
      \draw[black,postaction={decorate}] (0,0) -- (0,2)
        node[midway, left] {$\mathfrak{sp}$};
      \draw[black,postaction={decorate}] (0,2) -- (2,2)
        node[midway, above] {$\mathfrak{rs}$};
      \draw[fill] (0,0) node[left]  {$\mathfrak{p}$} circle[radius=1pt];
      \draw[fill] (2,0) node[right] {$\mathfrak{q}$} circle[radius=1pt];
      \draw[fill] (2,2) node[right] {$\mathfrak{r}$} circle[radius=1pt];
      \draw[fill] (0,2) node[left]  {$\mathfrak{s}$} circle[radius=1pt];
    \end{tikzpicture}
  \end{minipage}
	 \hspace*{-6.5cm}
  \begin{minipage}{0.4\textwidth}
    \centering
    \begin{tikzpicture}[
      x=1.6cm, y=1.6cm,
      every node/.style={font=\small,align=center}, 
      decoration={markings, mark=at position 0.5 with {\arrow{latex}}}
      ]
      \draw[black,postaction={decorate}] (0,0) -- (2,0)
        node[midway, below]
          {$\begin{array}{c}(j+1,k)(j,k)\\ \in E_1^+\end{array}$};
      \draw[black,postaction={decorate}] (2,0) -- (2,2)
        node[midway, right]
          {$\begin{array}{c}(j+1,k+1)(j+1,k)\\ \in E_2^+\end{array}$};
      \draw[black,postaction={decorate}] (0,0) -- (0,2)
        node[midway, left]
          {$\begin{array}{c}(j,k+1)(j,k)\\ \in E_2^+\end{array}$};
      \draw[black,postaction={decorate}] (0,2) -- (2,2)
        node[midway, above]
          {$\begin{array}{c}(j+1,k+1)(j,k+1)\\ \in E_1^+\end{array}$};
      \draw[fill] (0,0) node[left]  {$(j,k)$} circle[radius=1pt];
      \draw[fill] (2,0) node[right] {$(j+1,k)$} circle[radius=1pt];
      \draw[fill] (2,2) node[right] {$(j+1,k+1)$} circle[radius=1pt];
      \draw[fill] (0,2) node[left]  {$(j,k+1)$} circle[radius=1pt];
    \end{tikzpicture}
  \end{minipage}
  \caption{Left: vertices and oriented edges of a quadrilateral, 
    right: identification with pairs of integers in the domain.}
  \label{fig:two-diagrams}
\end{figure}

We define the discrete bundle $\underline{\Hh} := V \times \Hh$,
and now we review discrete gauge theory and define flat connections, see for example \cite{B, BCHPR, BHRS, NORYPJ}.

\begin{definition}
	A \textbf{section} of $\underline{\Hh}$ is a map $\sigma\colon V \to \Hh$.
	A section $\Phi$ of $\ubar{\Hh}$ is a \textbf{frame} if it is invertible: $\Phi\colon V \to \Hh^{\times}(=\Hh \setminus \{0\})$.
	A \textbf{discrete connection} on $\underline{\Hh}$ is a map $\eta\colon E\to \Hh^{\times}$ such that $\eta_{\mathfrak{ip}} = \eta_{\mathfrak{pi}}\inv$ for all $\mathfrak{pi}\in E$.
	Any frame $\Phi$ induces a connection $\eta$ via the formula
	\begin{equation}\label{eq:induced-connection}
		\eta_{\mathfrak{ip}} = \Phi(\mathfrak{i})\Phi(\mathfrak{p})\inv\quad \forall \mathfrak{ip}\in E.
	\end{equation}
	A section $\sigma$ is \textbf{parallel} for $\eta$ if $\sigma_{\mathfrak{i}} = \eta_{\mathfrak{ip}} \sigma_{\mathfrak{p}}$ for all edges $\mathfrak{ip}$. 
	A \textbf{parallel frame} for $\eta$ is a frame $\Phi\colon V\to \Hh^{\times}$ which is parallel for $\eta$,
	and any frame is parallel for its induced connection.
	A discrete connection $\eta$ is \textbf{flat} if 
	\begin{equation}
		\eta_{\mathfrak{rq}}\eta_{\mathfrak{qp}} = \eta_{\mathfrak{rs}}\eta_{\mathfrak{sp}}
	\end{equation}
	 on any face $\mathfrak{pqrs}\in F$.
\end{definition}

It is immediately seen that the following statements are equivalent:
\begin{enumerate}
		\item $\eta$ admits a parallel frame.
		\item $\eta$ is flat.
\end{enumerate}

\begin{definition}
	The collection of frames $G \colon V \to \Hh^{\times}$ forms a group under point-wise multiplication.
	This group is called the \textbf{gauge group} of $\ubar{\Hh}$. 
	It acts on sections $\sigma$ via left multiplication and acts on connections via
	\begin{equation}\label{eq:gauged-connection-formula}
		(G\cdot\eta)_{\mathfrak{ip}} := G_{\mathfrak{i}} \eta_{\mathfrak{ip}} G_{\mathfrak{p}}\inv.
	\end{equation}
	for all edges $\mathfrak{ip} \in E$.
	\end{definition}

  Note that $G\sigma$ is parallel for $G\cdot\eta$ if and only if $\sigma$ is parallel for $\eta$.
	A trivializing gauge for a given connection $\eta$ is a gauge $G$ such that $G\cdot\eta \equiv 1$. 
	Also, if $G$ is a trivializing gauge for $\eta$, 
	then, for any constant $\Phi_{0}$, $\Phi:=G\inv \Phi_{0}$ 
	is a parallel frame for $\eta$, 
  and hence $\eta$ is flat. 
	Conversely if $\eta$ is flat, the inverse of a parallel frame is a trivializing gauge.

We now restrict the choices for $V$ so that we can 
identify each vertex with two integers via
$$\mathfrak{p}\in V \leftrightarrow (j,k) \in \mathcal{D} \subset \Zz^2$$
where $\mathcal{D}$ is a domain of $\Zz^2$, see Figure \ref{fig:two-diagrams}.
That is, each face $\mathfrak{pqrs}$ is identified as
$$\mathfrak{p} \leftrightarrow (j,k),\;\;\mathfrak{q} \leftrightarrow (j+1,k),\;\;\mathfrak{r} \leftrightarrow (j+1,k+1),\;\; \mathfrak{s} \leftrightarrow (j,k+1).$$
Moreover, we can decompose $E=E_1^+ \cup E_1^- \cup E_2^+ \cup E_2^-$ with
$$E_1^+=\{(j+1,k)(j,k)\in E\},\;\;E_1^-=\{(j-1,k)(j,k)\in E\},$$
$$E_2^+=\{(j,k+1)(j,k)\in E\},\;\;E_2^-=\{(j,k-1)(j,k)\in E\}.$$

\begin{definition}
	Let $\eta\colon E \to \Hh^{\times}$ be a connection. 
	The \textbf{Lax pair} $(L,M)$ is defined on $E_1^{\pm}$ (for $L$) and
  on $E_2^{\pm}$ (for $M$) by 
	\begin{equation}
		L_{e_1} := \eta_{e_1},\;\; M_{e_2} := \eta_{e_2}.
	\end{equation} 
	 Then the flatness of $\eta$ is rewritten by 
		\begin{equation} \label{eq:compa}
			M_{\mathfrak{rq}} L_{\mathfrak{qp}} = L_{\mathfrak{rs}}M_{\mathfrak{sp}},
		\end{equation}
	  which is called the \textbf{compatibility condition}.
\end{definition}

\subsection{Circular nets and their flat connections}

\subsubsection{Sym-type formulas}
For $\Ss^2$ the unit sphere centered at the origin in $\Rr^3$, we call a pair of maps $(\x, \n) : V \to \Rr^3 \times \Ss^2$ a \textbf{contact element net}.
Let $\xi\in\Rr^{\times}$ and $\tau\in\Rr$. For any $t$-differentiable family $\Phi=(\Phi_t)_{t\in \Rr}$ of frames on $\ubar{\Hh}$, consider the \textbf{Sym-type} formulas (which we will abbreviate to ``Sym formula")
\begin{equation}\label{eq:sym-bob}
	\x := \xi\left[\Phi\inv \partial_t\Phi\right]^{\tr = 0} + \tau \n ,\quad \n := \Phi\inv(-i\boldsymbol{\sigma}_3)\Phi.
\end{equation}
Then for each $t$, $(\x,\n)$ is a contact element net.

\subsubsection{Ridid motions}
For contact element nets,   
a transformation $(\x,\n) \mapsto (\tilde{\x},\tilde{\n})$ defined by
$$\tilde{\x} = R\inv \x R + T,\;\; \tilde{\n} = R\inv \n R$$
for constants $R, T \in \Hh$ is a \textbf{rigid motion}.
Let $\Phi\colon V\to \Hh^{\times}$ be a frame with connection $\eta$ and let $H(t)\in\Hh^{\times}$ for all $t\in \Rr$ ($H$ is constant with respect to $V$). 
Then $\tilde{\Phi} := \Phi H$ is another frame with the same associated connection $\eta$. Let $(\x, \n)$ be the contact element net induced by $\Phi$ via the Sym formula \eqref{eq:sym-bob}. 
Then the new frame $\tilde{\Phi}$ also induces a contact element net $(\tilde{\x}, \tilde{\n})$ given by
\begin{equation} \label{eq:isom}
	\tilde{\x} = H\inv \x H + \xi \left[H\inv \partial_tH\right]^{\tr = 0},\;\; \tilde{\n} = H\inv \n H.
\end{equation}
Therefore, changing the initial condition chosen for the parallel frame $\Phi$ induces 
a rigid motion for $(\x,\n)$. 
\subsubsection{Gauging}

\begin{definition}
	A gauge $G$ is \textbf{admissible} if there exists a family $x_t\colon V \to \Rr$ and a section $y\colon V \to \Rr$ such that
	\begin{equation}
		G = e^{x_t}\matrix{e^{i y}}{0}{0}{e^{-iy}} = \exp(x_t\boldsymbol{\sigma}_0 + i y\boldsymbol{\sigma}_3).
	\end{equation}
\end{definition}

For an admissible gauge $G$, we compute the gauged net:
	\begin{equation*}\label{eq:gauged-immersion}
		\hat{\x} = \x + \xi \Phi_t\inv\left[G\inv \partial_tG\right]^{\tr = 0}\Phi_t =\x,
	\end{equation*}
	\begin{equation*}\label{eq:gauged-normal}
		\hat{\n} = \Phi\inv G\inv (-i \boldsymbol{\sigma}_{3}) G \Phi=\n,
	\end{equation*}
and we have the following lemma.

\begin{lemma}[Gauge invariance]\label{lemma:gauge-invariance}
	Let $\Phi_t$ be a family of frames and let $(\x,\n)$ be the contact element 
	net given by the Sym formula. Let $G$ be an admissible gauge, let $\hat{\Phi}_t := G\Phi_t$, and let $(\hat{\x}, \hat{\n})$ be the net induced by $\hat{\Phi}_t$.
	Then $(\x,\n) = (\hat{\x},\hat{\n})$.
\end{lemma}

\subsubsection{Edge-constraint nets and circular nets}
We give notations for the difference and summation operators acting on
functions of two discrete variables. For a function $F:\Zz^2 \to \Rr^n$ or $\Cc^n$,
$$\delta_j F(j,k):=F(j+1,k)-F(j,k),\;\; \delta_k F(j,k):=F(j,k+1)-F(j,k),$$
$$\sigma_j F(j,k):=F(j+1,k)+F(j,k),\;\; \sigma_k F(j,k):=F(j,k+1)+F(j,k).$$
In the case of a function of one discrete variable, we can abbreviate these to $\delta$ and $\sigma$.

We will define edge-constraint nets and their curvatures as in \cite{HS,HSW}.

\begin{definition} \label{definition:ec-net}
	A contact element net $(\x, \n)$ is an \textbf{edge-constraint net} (abbr. \textbf{ec-net}) 
	if, for $\star= j, k$,
	\begin{subequations} \label{eq:ec}
	\begin{equation}
		\scal{\delta_{\star} \x(j,k), \sigma_{\star} \n(j,k)}=0,
	\end{equation}
	 and  
		\begin{equation} \label{eq:degenerate1}
		\delta_{\star} \x(j,k) \neq 0,\;\;
    \sigma_{\star} \n(j,k) \neq 0
    \end{equation}
		for all edges $(j+1,k)(j,k)$ and $(j,k+1)(j,k)$, and 
    \begin{equation} \label{eq:degenerate2}
			\det (\x(j+1,k+1)-\x(j,k),\x(j+1,k)-\x(j,k+1),N)\neq 0
		\end{equation}
	\end{subequations}
		for all faces of $\x$, where $N$ is a unit normal vector such that 
		$$N \perp \textnormal{span}\{\n(j+1,k+1)-\n(j,k),\n(j+1,k)-\n(j,k+1)\}.$$
\end{definition}

\begin{definition}
	For every face of an ec-net $(\x,\n)$, the \textbf{Gaussian curvature} $K$ is
	defined as
	\begin{equation} \label{eq:Gaussian curvature}
		K:=\frac{\det (\n(j+1,k+1)-\n(j,k),\n(j+1,k)-\n(j,k+1),N)}{\det (\x(j+1,k+1)-\x(j,k),\x(j+1,k)-\x(j,k+1),N)}
	\end{equation}
	and \textbf{mean curvature} $H$ is defined as
	\begin{multline} \label{eq:mean curvature}
		H:=\frac{1}{2}\frac{\det (\x(j+1,k+1)-\x(j,k),\n(j+1,k)-\n(j,k+1),N)}{\det (\x(j+1,k+1)-\x(j,k),\x(j+1,k)-\x(j,k+1),N)}\\
		+\frac{1}{2}\frac{\det (\n(j+1,k+1)-\n(j,k),\x(j+1,k)-\x(j,k+1),N)}{\det (\x(j+1,k+1)-\x(j,k),\x(j+1,k)-\x(j,k+1),N)},
	\end{multline}
	where $N$ is a unit normal vector as in Definition \ref{definition:ec-net}.

  We will abbreviate the cases of constant Gaussian curvature $K$ and constant mean curvature $H$ as
 CGC $K$ and CMC $H$, respectively.
\end{definition}

Furthermore, to consider discrete surfaces with curvature line coordinates, 
we define circular nets, see for example \cite{BPW, BHR, BHRS, BP2, BS, NORYPJ, HSW}.

\begin{definition} \label{definition:c-net}
	  A map $\x : V \to \Rr^3$ is a \textbf{circle-circumscribed net} (abbr. \textbf{cc-net}) 
		if every face of $\x$ has all vertices lying on a circle.
    A contact element $(\x,\n)$ is a \textbf{circular net} (abbr. \textbf{c-net})  if:
    \begin{itemize}
		\item[\textnormal{(1)}] $(\x,\n)$ is an ec-net and $\x$ is a cc-net, and 
    \item[\textnormal{(2)}] $\delta_{\star} \x(j,k) \parallel \delta_{\star} \n(j,k)$ for $\star=j,k$.
    \end{itemize}
    A cc-net $\x$ is \textbf{locally embedded} if 
    \begin{multline} \label{eq:embeddedness}
			\text{cr}(\x(j,k), \x(j+1,k), \x(j+1,k+1), \x(j,k+1))\\
			=(\x(j,k)-\x(j+1,k))(\x(j+1,k)-\x(j+1,k+1))\inv\\
			\cdot(\x(j+1,k+1)-\x(j,k+1))(\x(j,k+1)-\x(j,k))\inv<0,
		\end{multline}
		and a c-net is a \textbf{discrete isothermic net} if 
		$$\text{cr}(\x(j,k), \x(j+1,k), \x(j+1,k+1), \x(j,k+1))=\frac{a_{(j+1,k)(j,k)}}{b_{(j,k+1)(j,k)}}<0$$
		for all faces, where $a$ and $b$, defined on edges, are an edge-labeling, i.e.
		$$a_{(j+1,k)(j,k)}=a_{(j+1,k+1)(j,k+1)},\;\;b_{(j,k+1)(j,k)}=b_{(j+1,k+1)(j+1,k)}.$$
		We call $a$ and $b$ \textbf{cross ratio factorizing functions}.
\end{definition}

Figure \ref{figure:embedded} shows examples of locally embedded and non-embedded faces.

\begin{figure} 
\setlength{\tabcolsep}{35pt}
\begin{tabular}{cc}
\begin{tikzpicture}[scale=0.8, line cap=round, line join=round]
  \def\r{2.2}
  \coordinate (O) at (0,0);

  \coordinate (A) at ({\r*cos(20)},{\r*sin(20)}); 
  \coordinate (B) at ({\r*cos(130)},{\r*sin(130)});
  \coordinate (C) at ({\r*cos(210)},{\r*sin(210)}); 
  \coordinate (D) at ({\r*cos(350)},{\r*sin(350)});

  \fill[gray!40] (A)--(B)--(C)--(D)--cycle;

  \draw (O) circle (\r);

  \draw (D)--(A);
  \draw (A)--(B);
  \draw (B)--(C);
  \draw (D)--(C);

  \fill (A) circle (2.2pt);
	\fill (B) circle (2.2pt);
  \fill (D) circle (2.2pt);
  \fill (C) circle (2.2pt);
\end{tikzpicture}&
\begin{tikzpicture}[scale=0.8, line cap=round, line join=round]
  \def\r{2.2}
  \coordinate (O) at (0,0);

  \coordinate (A) at ({\r*cos(60)},{\r*sin(60)}); 
  \coordinate (B) at ({\r*cos(160)},{\r*sin(160)});
  \coordinate (C) at ({\r*cos(230)},{\r*sin(230)}); 
  \coordinate (D) at ({\r*cos(330)},{\r*sin(330)});

  \fill[gray!40] (A)--(B)--(D)--(C)--cycle;

  \draw (O) circle (\r);

  \draw (A)--(B);
  \draw (B)--(D);
  \draw (D)--(C);
  \draw (C)--(A);

  \fill (A) circle (2.2pt);
	\fill (B) circle (2.2pt);
  \fill (D) circle (2.2pt);
  \fill (C) circle (2.2pt);
\end{tikzpicture}
\end{tabular}
\caption{Left: a face of a locally embedded cc-net, right: a face of a non-embedded cc-net.}
\label{figure:embedded}
\end{figure}
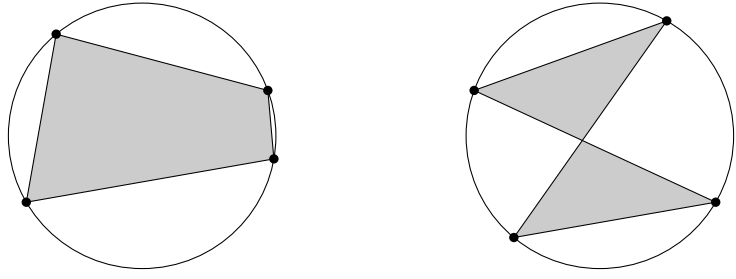

\subsubsection{CGC and CMC circular nets}
We will now compare choices for flat connections for constant Gaussian curvature circular nets (both with $K>0$ and $K<0$), which involves looking
at functions defined on vertices or edges, and at spectral parameters and compatibility conditions.
In the case of constant negative Gaussian curvature, we will have two distinct choices, each with its own advantages, as we will see later.

\paragraph{\textbf{Positive CGC circular nets and CMC isothermic nets}:}

In \cite{BP3,PW}, positive CGC circular nets with $K=1$ are made with $\xi=-2$ and $\tau=0$ 
and their corresponding CMC isothermic nets are made with $\xi=-2$ and $\tau=\pm1$ in \eqref{eq:sym-bob}. 
For two different edges $e_1 \in E_1^{\pm}$, $e_2 \in E_2^{\pm}$, 
we will use a connection that takes the form
\begin{subequations} \label{eq:connection-CMC}
\begin{equation}
	\eta_{e_1} = \frac{1}{\alpha}\matrix{\vfrak}{-e^{i t} u - e^{-i t} u\inv}{{e^{-i t}} {u} + e^{i t} {u}\inv}{\bar{\vfrak}},\;\; \vfrak\colon E_1^{\pm} \to \Cc,\;\; u\colon E_1^{\pm} \to \Rr^{\times},
\end{equation}
\begin{equation}
  \eta_{e_2}= \frac{1}{\beta}\matrix{\hfrak}{- i e^{i t} v + i e^{-i t} v\inv}{i e^{i t} v\inv - i e^{-i t} v}{\bar{\hfrak}},\;\; \hfrak\colon E_2^{\pm} \to \Cc, \;\; v\colon E_2^{\pm} \to \Rr^{\times},
\end{equation}
\end{subequations}
where 
\begin{equation}
	\alpha^2 = \abs{\vfrak}^2 + u^2 + u^{-2} + e^{2 i t} + e^{-2 i t },\;\;
  \beta^2=\abs{\hfrak}^2 + v^2 + v^{-2} - e^{2 i t} - e^{-2 i t }
\end{equation}
and $\alpha$, $\beta$ coincide on opposite edges, that is,
$$\alpha_{\mathfrak{qp}}=\alpha_{\mathfrak{rs}},\;\;\beta_{\mathfrak{sp}}=\beta_{\mathfrak{rq}}$$
for all faces $\mathfrak{pqrs} \in F$.
Note that the equations for $\alpha$, $\beta$ make $\eta$ a unitary quaternion. 
With regard to the orientation, we have
\begin{equation}
\alpha_{\mathfrak{pq}} = -\alpha_{\mathfrak{qp}},\;\;u_{\mathfrak{pq}} = u_{\mathfrak{qp}},\;\; \vfrak_{\mathfrak{pq}}=-\overline{\vfrak_{\mathfrak{qp}}}
\end{equation}
for all edges $\mathfrak{qp} \in E_{1}^{+}$, $\mathfrak{pq} \in E_{1}^{-}$, and
\begin{equation}
\beta_{\mathfrak{ps}}=-\beta_{\mathfrak{sp}},\;\;v_{\mathfrak{ps}}= v_{\mathfrak{sp}},\;\; \hfrak_{\mathfrak{ps}}=-\overline{\hfrak_{\mathfrak{sp}}}
\end{equation}
for all edges $\mathfrak{sp} \in E_{2}^{+}$, $\mathfrak{ps} \in E_{2}^{-}$.
We note that the cross ratio for each face of $\x$ for a CMC isothermic net at $t=0$ is $-\beta^2/\alpha^2$ (see \cite{BP3}),
thus $\alpha^2$ and $-\beta^2$ are cross ratio factorizing functions.

\begin{lemma}[\cite{BP3}]\label{lemma:compatibility1}
    The connection $\eta$ of the form \eqref{eq:connection-CMC}
	is flat for all $t$ if and only if
	\begin{subequations}\label{eq:compatibilitycmc}
		\begin{equation}\label{eq:compatibilitycmc1}
			u_{\mathfrak{qp}}u_{\mathfrak{rs}} = v_{\mathfrak{sp}}v_{\mathfrak{rq}},
		\end{equation}
		\begin{equation}\label{eq:compatibilitycmc2}
			\hfrak_{\mathfrak{rq}}\vfrak_{\mathfrak{qp}}-\vfrak_{\mathfrak{rs}}\hfrak_{\mathfrak{sp}} \\
            =i (u_{\mathfrak{rs}}v_{\mathfrak{sp}} - (u_{\mathfrak{rs}}v_{\mathfrak{sp}})\inv \\
            + v_{\mathfrak{rq}}u_{\mathfrak{qp}} - (v_{\mathfrak{rq}}u_{\mathfrak{qp}})\inv),
		\end{equation}
		\begin{equation}\label{eq:compatibilitycmc3}
			i\;v_{\mathfrak{rq}}\overline{\vfrak_{\mathfrak{qp}}} - u_{\mathfrak{rs}}\overline{\hfrak_{\mathfrak{sp}}} \\
            +\hfrak_{\mathfrak{rq}}u_{\mathfrak{qp}} - i\;\vfrak_{\mathfrak{rs}}v_{\mathfrak{sp}} = 0,
		\end{equation}
		\begin{equation}\label{eq:compatibilitycmc4}
			i\;v_{\mathfrak{sp}}\overline{\vfrak_{\mathfrak{qp}}} + u_{\mathfrak{qp}} \overline{\hfrak_{\mathfrak{sp}}} \\
            -\hfrak_{\mathfrak{rq}}u_{\mathfrak{rs}} - i\;\vfrak_{\mathfrak{rs}}v_{\mathfrak{rq}} = 0
		\end{equation}
	\end{subequations}
	for all faces $\mathfrak{pqrs} \in F$.
\end{lemma}

When $u, v>0$ (see \cite{BP3, PW}), we can rewrite $u$ and $v$ 
in exponential form in terms of a function defined on vertices. Defining $\rho : V \to \Rr$ by 
\begin{equation} \label{eq:vertexfunc}
	u_{\mathfrak{qp}}=e^{\frac{1}{2}(\rho(\mathfrak{q})+\rho(\mathfrak{p}))},\;\;
	v_{\mathfrak{sp}}=e^{\frac{1}{2}(\rho(\mathfrak{s})+\rho(\mathfrak{p}))},
\end{equation}
\eqref{eq:compatibilitycmc1} is satisfied automatically, and 
\eqref{eq:compatibilitycmc2}, \eqref{eq:compatibilitycmc3}, \eqref{eq:compatibilitycmc4} become
\begin{subequations} \label{eq:discrete sinh gordon}
	\begin{equation}
		\hfrak_{\mathfrak{rq}}\vfrak_{\mathfrak{qp}}-\vfrak_{\mathfrak{rs}}\hfrak_{\mathfrak{sp}}\\
		=4 i  \sinh \frac{\rho(\mathfrak{r})+\rho(\mathfrak{q})+\rho(\mathfrak{s})+\rho(\mathfrak{p})}{2} \cosh \frac{\rho(\mathfrak{s})-\rho(\mathfrak{q})}{2},
	\end{equation}
	\begin{equation}
		\hfrak_{\mathfrak{rq}}=
		\frac{1}{\cosh \frac{\rho(\mathfrak{s})-\rho(\mathfrak{q})}{2}}\Bigg(\overline{\hfrak_{\mathfrak{sp}}} \cosh \frac{\rho(\mathfrak{r})-\rho(\mathfrak{p})}{2}\\
		- i \overline{\vfrak_{\mathfrak{qp}}}\sinh \frac{\rho(\mathfrak{r})-\rho(\mathfrak{s})+\rho(\mathfrak{q})-\rho(\mathfrak{p})}{2}\Bigg),
	\end{equation}
	\begin{equation}
		\vfrak_{\mathfrak{rs}}=
		\frac{1}{\cosh \frac{\rho(\mathfrak{s})-\rho(\mathfrak{q})}{2}}\Bigg(i \overline{\hfrak_{\mathfrak{sp}}} \sinh \frac{\rho(\mathfrak{r})-\rho(\mathfrak{q})+\rho(\mathfrak{s})-\rho(\mathfrak{p})}{2}\\
		+ \overline{\vfrak_{\mathfrak{qp}}} \cosh \frac{\rho(\mathfrak{r})-\rho(\mathfrak{p})}{2}\Bigg),
	\end{equation}
\end{subequations}
which are shown to be a discretized version of the elliptic sinh-Gordon equation in \cite{PW}.

Building upon \cite{HSW}, the work \cite{HSS} defines 4D cross ratio systems and provides  
another form for Lax pairs. 

\paragraph{\textbf{Negative CGC circular nets (cK-nets)}:}
In \cite{HS}, the Lax pairs for circular nets with Gaussian curvature 
$K=-1$ are made with $\xi=2$ and $\tau = 0$ in the Sym formula 
\eqref{eq:sym-bob}. The Lax pair takes the form
\begin{subequations} \label{eq:Lax}
 \begin{equation}\label{eq:Lax1}
  \mathcal{L}_{\mathfrak{qp}} =\matrix{\ell( s(\mathfrak{p}) \inv \cot \frac{\delta_{(1)}}{2} +  s(\mathfrak{q}) \tan \frac{\delta_{(1)}}{2} )}
     {i(e^t - e^{-t} s(\mathfrak{p})s(\mathfrak{q}))}
     {i(e^t - e^{-t} (s(\mathfrak{p})s(\mathfrak{q}))\inv)}
     {\ell\inv( s(\mathfrak{p}) \cot \frac{\delta_{(1)}}{2} +  s(\mathfrak{q}) \inv \tan \frac{\delta_{(1)}}{2} )},
 \end{equation}
 \begin{equation}\label{eq:Lax2}
  \mathcal{M}_{\mathfrak{sp}} =\matrix{m( s(\mathfrak{p}) \inv \cot \frac{\delta_{(2)}}{2} +  s(\mathfrak{s}) \tan \frac{\delta_{(2)}}{2} )}
     {i(e^t - e^{-t} s(\mathfrak{p})s(\mathfrak{s}))}
     {i(e^t - e^{-t} (s(\mathfrak{p})s(\mathfrak{s}))\inv)}
     {m\inv( s(\mathfrak{p}) \cot \frac{\delta_{(2)}}{2} +  s(\mathfrak{s}) \inv \tan \frac{\delta_{(2)}}{2} )},
 \end{equation}
\end{subequations}
where $s \colon V \to \Ss^1$ and the functions 
$\ell$, $m$, $\delta_{(1)}$, $\delta_{(2)}$ are defined on edges. 
Here for $i=1$ or $i=2$, if $|\sin \delta_{(i)}|\leq1$, then $\ell$ 
or $m$ is unitary, respectively. If $|\sin\delta_{(i)}|>1$, then 
the absolute value of $l$ or $m$ is determined so that \eqref{eq:Lax1} 
or \eqref{eq:Lax2} is a quaternion. To make the orientation conditions 
$L_{\mathfrak{qp}}=L_{\mathfrak{pq}}\inv$ 
and $M_{\mathfrak{sp}}=M_{\mathfrak{ps}}\inv$ explicit,
and for considering flatness of $L$ and $M$,
and for making comparison with the positive CGC and CMC cases above, 
for two different edges $e_1 \in E_1^{\pm}$, $e_2 \in E_2^{\pm}$, 
we consider a different connection using only edge variables as in Remark 15 of \cite{HS}:

\begin{subequations}\label{eq:connection-CGC}
\begin{equation}
 \eta_{e_1} =L_{e_1} = \frac{1}{\alpha}\matrix{\vfrak}{e^{-t} u - e^t u\inv}{e^t u - e^{-t} u\inv}{\bar{\vfrak}},\;\; \vfrak\colon E_1^{\pm} \to \Cc,\;\; u\colon E_1^{\pm} \to \Ss^1,
\end{equation}
\begin{equation}
 \eta_{e_2} =M_{e_2}= \frac{1}{\beta}\matrix{\hfrak}{e^{-t} v - e^t v\inv}{e^t v - e^{-t} v\inv}{\bar{\hfrak}},\;\; \hfrak\colon E_2^{\pm} \to \Cc, \;\; v\colon E_2^{\pm} \to \Ss^1,
\end{equation}
\end{subequations}
where
\begin{equation}
	\alpha^2 = \abs{\vfrak}^2 - u^2 -u^{-2} + e^{2t} +e^{-2t},\;\;
	\beta^2=\abs{\hfrak}^2 - v^2 -v^{-2} + e^{2t} +e^{-2t}
\end{equation}
and $\alpha$, $\beta$ coincide on opposite edges. 
The functions $\alpha$ and $\beta$ are now not cross ratio factorizing functions, but they 
do correspond to $\delta_{(1)}$ and $\delta_{(2)}$, respectively, as the proof of
Proposition \ref{prop:timandandrewlaxpair} at the end of this section shows.
We set
\begin{equation}
	\alpha_{\mathfrak{pq}} = -\alpha_{\mathfrak{qp}},\;\;u_{\mathfrak{pq}} = u_{\mathfrak{qp}},\;\;\vfrak_{\mathfrak{pq}}=-\overline{\vfrak_{\mathfrak{qp}}},
\end{equation}
\begin{equation}
	\beta_{\mathfrak{ps}}=-\beta_{\mathfrak{sp}},\;\;v_{\mathfrak{ps}}=v_{\mathfrak{sp}},\;\;\hfrak_{\mathfrak{ps}}=-\overline{\hfrak_{\mathfrak{sp}}}.
\end{equation}
Then $\eta$ satisfies the conditions to be a unitary connection, for all $t$,
and we say more about $\eta$ in the next lemma and proposition.

\begin{lemma}\label{lemma:compatibility2}
    The connection $\eta$ of the form \eqref{eq:connection-CGC}
	is flat for all $t\in\Rr$ if and only if
	\begin{subequations}\label{eq:compatibilitycgc}
		\begin{equation}\label{eq:compatibilitycgc1}
			u_{\mathfrak{qp}}u_{\mathfrak{rs}} = v_{\mathfrak{sp}}v_{\mathfrak{rq}},
		\end{equation}
		\begin{equation}\label{eq:compatibilitycgc2}
			\hfrak_{\mathfrak{rq}}\vfrak_{\mathfrak{qp}}-\vfrak_{\mathfrak{rs}}\hfrak_{\mathfrak{sp}} \\
            = u_{\mathfrak{rs}}v_{\mathfrak{sp}} + (u_{\mathfrak{rs}}v_{\mathfrak{sp}})\inv \\
            - v_{\mathfrak{rq}}u_{\mathfrak{qp}} - (v_{\mathfrak{rq}}u_{\mathfrak{qp}})\inv,
		\end{equation}
		\begin{equation}\label{eq:compatibilitycgc3}
			v_{\mathfrak{rq}}\overline{\vfrak_{\mathfrak{qp}}} - u_{\mathfrak{rs}}\overline{\hfrak_{\mathfrak{sp}}} \\
            +\hfrak_{\mathfrak{rq}}u_{\mathfrak{qp}} - \vfrak_{\mathfrak{rs}}v_{\mathfrak{sp}} = 0,
		\end{equation}
		\begin{equation}\label{eq:compatibilitycgc4}
			v_{\mathfrak{sp}} \overline{\vfrak_{\mathfrak{qp}}} - u_{\mathfrak{qp}} \overline{\hfrak_{\mathfrak{sp}}} \\
            +\hfrak_{\mathfrak{rq}}u_{\mathfrak{rs}} - \vfrak_{\mathfrak{rs}}v_{\mathfrak{rq}} = 0.
		\end{equation}
	\end{subequations}
\end{lemma}
\begin{proof}
	The computation is the same as Lemma \ref{lemma:compatibility1}.
\end{proof}

Similarly to \eqref{eq:vertexfunc}, rewriting $u$, $v$ in terms of a function defined on vertices, 
\eqref{eq:compatibilitycgc} can be rewritten in a form similar to \eqref{eq:discrete sinh gordon}. When we define $\rho : V \to \Rr$ as
\begin{equation} \label{eq:s}
	u_{\mathfrak{qp}}=e^{\frac{1}{2}i(\rho(\mathfrak{q})+\rho(\mathfrak{p}))},\;\;
	v_{\mathfrak{sp}}=e^{\frac{1}{2}i(\rho(\mathfrak{s})+\rho(\mathfrak{p}))},\;\;
\end{equation}
\eqref{eq:compatibilitycgc1} is satisfied automatically, and 
\eqref{eq:compatibilitycgc2}, \eqref{eq:compatibilitycgc3}, \eqref{eq:compatibilitycgc4} become
\begin{subequations} \label{eq:discrete sin gordon}
	\begin{equation}
		\hfrak_{\mathfrak{rq}}\vfrak_{\mathfrak{qp}}-\vfrak_{\mathfrak{rs}}\hfrak_{\mathfrak{sp}}
		=4 \sin \frac{\rho(\mathfrak{r})+\rho(\mathfrak{s})+\rho(\mathfrak{q})+\rho(\mathfrak{p})}{2}\sin \frac{\rho(\mathfrak{s})-\rho(\mathfrak{q})}{2},
	\end{equation}
	\begin{equation}
		\hfrak_{\mathfrak{rq}}=
		\frac{1}{\sin \frac{\rho(\mathfrak{q})-\rho(\mathfrak{s})}{2}}\Bigg(\overline{\hfrak_{\mathfrak{sp}}} \sin \frac{\rho(\mathfrak{r})-\rho(\mathfrak{p})}{2}\\
		- \overline{\vfrak_{\mathfrak{qp}}} \sin \frac{\rho(\mathfrak{r})-\rho(\mathfrak{s})+\rho(\mathfrak{q})-\rho(\mathfrak{p})}{2}\Bigg),
	\end{equation}
	\begin{equation}
		\vfrak_{\mathfrak{rs}}=
		\frac{1}{\sin \frac{\rho(\mathfrak{s})-\rho(\mathfrak{q})}{2}}\Bigg(-\overline{\hfrak_{\mathfrak{sp}}} \sin \frac{\rho(\mathfrak{r})-\rho(\mathfrak{q})+\rho(\mathfrak{s})-\rho(\mathfrak{p})}{2}\\
		+ \overline{\vfrak_{\mathfrak{qp}}} \sin \frac{\rho(\mathfrak{r})-\rho(\mathfrak{p})}{2}\Bigg).
	\end{equation}
\end{subequations}

These equations are anologous to Equations \eqref{eq:discrete sinh gordon}, but we see cosine and sine functions rather than their hyperbolic cosine and sine counterparts.

\begin{proposition}\label{prop:timandandrewlaxpair}
	The connection \eqref{eq:connection-CGC} is gauge equivalent to \eqref{eq:Lax}.
\end{proposition}
\begin{proof}
	Let $s\colon V \to \Ss^1$ defined by 
	\begin{equation} \label{eq:ss}
    s:=e^{i \rho}
	\end{equation}
	for $\rho$ as in \eqref{eq:s}. 
	For an arbitrary initial value $g_{0}\in\Rr^{\times}$, 
	define $g\colon V \to \Rr^{\times}$ by
	\begin{equation*}
		g(\mathfrak{q}) = \alpha_{\mathfrak{qp}}g(\mathfrak{p}),\;\;
		g(\mathfrak{s}) = \beta_{\mathfrak{sp}}g(\mathfrak{p})
	\end{equation*}
	for edges $\mathfrak{qp} \in E_1^{+}$ and $\mathfrak{sp} \in E_2^{+}$,
	and let $G$ be the admissible gauge defined by
	\begin{equation}\label{eq:gauge-tim-andrew}
		G := g \matrix{\frac{\sqrt{s}}{\sqrt{i}}}{0}{0}{\frac{\sqrt{i}}{\sqrt{s}}},
	\end{equation}
  applied to a parallel frame for \eqref{eq:connection-CGC}.
	This produces a new Lax pair
		\begin{equation*}
		 \mathcal{L}_{\mathfrak{qp}}=G(\mathfrak{q}) L_{\mathfrak{qp}} G(\mathfrak{p})\inv = \matrix{\vfrak\frac{\sqrt{s(\mathfrak{q})}}{\sqrt{s(\mathfrak{p})}}}{i(e^{t} - s(\mathfrak{p}) s(\mathfrak{q})e^{-t})}
			{i(e^{t} - \frac{e^{-t}}{s(\mathfrak{p}) s(\mathfrak{q})})}{\bar{\vfrak}\frac{\sqrt{s(\mathfrak{p})}}{\sqrt{s(\mathfrak{q})}}},\quad 
		\end{equation*}
		\begin{equation*}
			\mathcal{M}_{\mathfrak{sp}}=G(\mathfrak{s}) M_{\mathfrak{sp}} G(\mathfrak{p})\inv = \matrix{\hfrak\frac{\sqrt{s(\mathfrak{s})}}{\sqrt{s(\mathfrak{p})}}}{i(e^{t} - s(\mathfrak{p}) s(\mathfrak{s})e^{-t})}
			{i(e^{t} - \frac{e^{-t}}{s(\mathfrak{p}) s(\mathfrak{s})})}{\bar{\hfrak}\frac{\sqrt{s(\mathfrak{p})}}{\sqrt{s(\mathfrak{s})}}}.
		\end{equation*}
    Choosing $\delta_{(1)}$, $\delta_{(2)}$, $\ell$, $m$ such that
   \begin{subequations} \label{eq:delta}
		\begin{equation} 
     \sin \delta_{(1)}=\frac{2}{\alpha|_{t=0}},\;\;\sin\delta_{(2)}=\frac{2}{\beta|_{t=0}},
		\end{equation}
		\begin{equation}
		 \ell=\frac{\vfrak}{u\inv \cot \frac{\delta_{(1)}}{2}+ u \tan\frac{\delta_{(1)}}{2}},
		 \;\;m=\frac{\hfrak}{v\inv \cot \frac{\delta_{(2)}}{2} + v \tan\frac{\delta_{(2)}}{2}}
	 \end{equation}
	\end{subequations}
	yields the Lax pair \eqref{eq:Lax}. 
	 Since $\alpha(t)$ and $\beta(t)$ coincide on opposite edges for all $t$, the same is so for $\sin \delta_{(1)}$ and $\sin \delta_{(2)}$. 
	 One can find that 
   if, for $i=1$ or $i=2$, $|\sin \delta_{(i)}|\leq1$, then $\ell$ or $m$ is unitary, respectively. If $|\sin \delta_{(i)}|>1$, the absolute values of $\ell$ and $m$ are given, like in Theorem 5 of \cite{HS}, by
  $$|\ell|^2=\frac{u \cot\frac{\delta_{(1)}}{2} + u \inv \tan \frac{\delta_{(1)}}{2}}{u \tan\frac{\delta_{(1)}}{2}+u\inv \cot \frac{\delta_{(1)}}{2}},\;\;|m|^2=\frac{v \cot\frac{\delta_{(2)}}{2}+v\inv \tan \frac{\delta_{(2)}}{2}}{v \tan\frac{\delta_{(2)}}{2}+v\inv \cot \frac{\delta_{(2)}}{2}}.$$
  \end{proof}

\section{Circular nets of revolution with constant Gaussian curvature}
 In \cite{SN}, a way of finding explicit parametrizations of rotational
 symmetric circular nets with constant Gaussian curvature was 
 introduced.
 In this section, we will review those parametrizations,
 which use discrete trigonometric functions and 
 discrete hyperbolic functions satisfying particular difference equations.
 Then, invoking Jacobi elliptic functions, we discuss the singularities of  
 those nets. This section lays groundwork for investigating the corresponding flat connections in Section 3.
 
\subsection{Circular nets of revolution}
 \begin{definition} \label{definition:r-net}
   A map $\x : V \to \Rr^3$ is called a \textbf{net of revolution} if  
   the motion $\x(j,k) \mapsto \x(j,k+1)$ (or $\x(j,k) \mapsto \x(j+1,k)$) represents a rotation about an axis $L$ 
   with rotation angle $\theta \in (-\pi,0) \cup (0,\pi)$ for all $(j,k)$.
   Furthermore, a contact element net $(\x,\n)$ is called a \textbf{contact element net of revolution (r-net)} if  
   additionally $\n(j,k) \mapsto \n(j,k+1)$ (or $\n(j,k) \mapsto \n(j+1,k)$, respectively)
   represents the rotation of the same angle $\theta$ about an axis
   through the origin parallel to $L$.
 \end{definition}

  \begin{lemma}\label{lemma:circular net of revolution}
    Assume that $\x$ is a cc-net and a net of revolution in the $k$ direction. 
    Then $\x$ is locally embedded
    if and only if the image of the map
    $$\x(\cdot,k):j \mapsto \x(j,k)$$ 
    is a planar curve whose edges neither vanish nor intersect the rotational axis. 
  \end{lemma}
  \begin{proof}
    Recalling Definition \ref{definition:c-net} of local embeddedness in Equation \eqref{eq:embeddedness}, all of its edges must be non-zero. 
    Up to rigid motion, we may assume that the axis of rotation is the $-i \boldsymbol{\sigma}_3$-axis.
    Without loss of generality, for one arbitrary quadrilateral we may assume
    $$\x(j,k)=(f,0,0),\;\;\x(j,k+1)=(f \cos \theta,f \sin \theta, 0),$$
    $$\x(j+1,k+1)=(f_1 \cos (\theta + \tau),f_1 \sin (\theta +\tau), h),\;\;\x(j+1,k)=(f_1 \cos \tau,f_1 \sin \tau, h)$$
    for some $f,f_1,h,\theta,\tau \in \Rr$. 
    Then 
    $$\text{cr}(\x(j,k), \x(j+1,k), \x(j+1,k+1), \x(j,k+1))<0$$
    is equivalent to $f\neq0$ and $f_1\neq0$, and
    \begin{equation*}
    h\sin \tau = (h^2+f_1^2-f^2)\sin \tau = 0,\;\; f f_1((f^2+h^2+f_1^2)\cos \tau - 2ff_1)>0.
    \end{equation*}
    These relations cannot hold if $\sin \tau \neq 0$, 
    because then $h=0$ and with $|\cos \tau| < 1$ we have a contradiction.
    Therefore we may assume $\sin \tau = 0$ and the relations simplify to
    \begin{equation*}
     \cos \tau =\pm 1,\;\;
     \text{sign}(f_1)=\text{sign}(f \cos \tau),\;\; \text{and}\;\;
     f_1 \neq f \cos \tau\;\; \text{or}\;\; h\neq 0.
    \end{equation*}
    Therefore, $\x(j,k)$ and $\x(j+1,k)$ lie on the plane $\text{span}\{-i \boldsymbol{\sigma}_1, -i \boldsymbol{\sigma}_3\}$ and the edge 
    connecting them is nonvanishing and does not intersect the $-i \boldsymbol{\sigma}_3$-axis. 
    Continuing this argument inductively to other values of $j$, 
    we find that $\x(\cdot,k)$ satisfies the conclusion of the lemma.
  \end{proof}
 
  Using Lemma \ref{lemma:circular net of revolution} and the definitions of c-nets (Definition \ref{definition:c-net}) and r-nets (Definition \ref{definition:r-net}), we have the following proposition.

  \begin{proposition} \label{proposition:rotational surface}
    Let $(\x,\n)$ be an r-net in the $k$ direction.
    Then $(\x,\n)$ is a c-net and $\x$ is locally embedded if and only if, up to rigid motion, $(\x,\n)$ can be parametrized by
    \begin{subequations} \label{eq:surf and normal}
     \begin{equation} \label{eq:surf}
      \x(j,k)=\Bigr(f(j)\cos \theta k, f(j)\sin \theta k, h(j)\Bigr),
     \end{equation}
     \begin{equation} \label{eq:normal}
      \n(j,k)=\Bigr(a(j)\cos \theta k, a(j)\sin \theta k, b(j)\Bigr), 
     \end{equation}
    \end{subequations}
   for $\theta \in (-\pi,0) \cup (0,\pi)$ and discrete functions $f$, $h$, $a$, and $b$
   satisfying
   \begin{subequations} \label{eq:embedded}
   \begin{equation} \label{eq:embeddeda}
   f(j)\neq 0,\;\; 
   \text{sign}(f(j+1)) = \text{sign}(f(j)),\;\;\text{and}\;\;
   \delta f(j) \neq 0\;\; \text{or} \;\; \delta h(j) \neq 0,
   \end{equation}
   \begin{equation} \label{eq:embeddedb}
    a(j)^2 + b(j)^2 = 1,\;\; \text{and}\;\;\sigma a(j) \neq 0\;\;  \text{or} \;\; \sigma b(j) \neq 0,
   \end{equation}
   \begin{equation}\label{eq:embeddedc}
     \vector{\delta f(j)}{\delta h(j)} \perp \vector{\sigma a(j)}{\sigma b(j)},\;\;\text{and}\;\;
     \vector{\delta f(j)}{\delta h(j)} \parallel \vector{\delta a(j)}{\delta b(j)}.
   \end{equation}
  \end{subequations}
  \end{proposition}

   \begin{definition} \label{definition:rc-net}
    We call a c-net $(\x,\n)$ parametrized by \eqref{eq:surf and normal} satisfying \eqref{eq:embedded}, up to a rigid motion,
    a \textbf{c-net of revolution (rc-net)}.
  \end{definition}

  \begin{remark} 
  In Proposition \ref{proposition:rotational surface}, 
  the conditions \eqref{eq:embedded} impose local embeddedness. 
 If we have a c-net $(\x,\n)$ parametrized by 
  \eqref{eq:surf and normal} without imposing local embeddedness,
   then $f$, $h$, $a$, and $b$ need satisfy only \eqref{eq:embeddedb}, \eqref{eq:embeddedc}, and $f(j) \neq 0$, and
   \begin{equation} \label{eq:non-degenerate}
   f(j+1) \neq -f(j),\;\;\text{and}\;\;
   \delta f(j) \neq 0\;\; \text{or} \;\; \delta h(j) \neq 0,
   \end{equation}
  where the new condition $f(j+1) \neq -f(j)$ comes from the non-degeneracy \eqref{eq:degenerate2} for faces of $\x$.
  \end{remark}

  From \eqref{eq:embeddedc} and $\vector{\sigma a}{\sigma b} \perp \vector{\delta a}{\delta b}$, we can check the following lemma.

  \begin{lemma} \label{lemma:profile}
    For an rc-net $(\x,\n)$ as in \eqref{eq:surf and normal},
    there exists a sequence $\{c(j)\}$ of nonzero terms which satisfies
    \begin{equation}\label{eq:edge1}
      \vector{\delta f(j)}{\delta h(j)}=c(j)\vector{\sigma b(j)}{-\sigma a(j)}=\matrix{0}{c(j)}{-c(j)}
      {0}\vector{\sigma a(j)}{\sigma b(j)}.
    \end{equation}
  \end{lemma}

  \subsection{rc-nets with constant Gaussian curvature} 
   One can check that, for rc-nets $(\x,\n)$ as in \eqref{eq:surf and normal} and $c(j)$ as in \eqref{eq:edge1}, 
   the Gaussian curvature is 
   \begin{equation} \label{eq:gauss}
    K(j)=\frac{a(j+1)^2-a(j)^2}{f(j+1)^2-f(j)^2}=-\frac{\delta b(j)}{c(j)\sigma f(j)}.
   \end{equation}

   Equation \eqref{eq:gauss} implies the following lemma.

  \begin{lemma} \label{lemma:gauss}
    Let $(\x,\n)$ be an rc-net as in \eqref{eq:surf and normal}. 
    If $K\equiv 1$, then
    there exists $\kappa > 0$ such that for all
    $j$, 
    \begin{equation} \label{eq:pgauss}
      f(j)^2-a(j)^2=f(j)^2+b(j)^2-1=\kappa^2-1.
    \end{equation}
    Rather, if $K\equiv -1$, then
    there exists $\kappa > 0$ such that for all
    $j$, 
    \begin{equation} \label{eq:ngauss}
      f(j)^2+a(j)^2=f(j)^2-b(j)^2+1=\frac{1}{\kappa^2}.
    \end{equation}
  \end{lemma}
  
   Let us recall the explicit parametrization of rc-nets with constant Gaussian curvature 
   found in \cite{SN}, now reconsidered
   using the $c(j)$ in Lemma \ref{lemma:profile}.

   \begin{theorem} \label{theorem:trig}
     Let $(\x,\n)$ be an rc-net as in \eqref{eq:surf and normal},
     with sequence $\{c(j)\}$ as in \eqref{eq:edge1}.
     Then $K \equiv 1$ if and only if $f(j)$ and $b(j)$ satisfy
     \begin{equation} \label{eq:soltrig}
      \left\{ \,
      \begin{aligned}
       &f(j)=A\cos{\Bigg(\bigg(\sum^{j-1}_{s=0}\theta(s)\bigg)+B\Bigg)},\\
       &b(j)=A\sin{\Bigg(\bigg(\sum^{j-1}_{s=0}\theta(s)\bigg)+B\Bigg)}
      \end{aligned}
      \right.
     \end{equation}
     for some $A$, $B \in \Rr$ and $c(j)$ such that \eqref{eq:soltrig} satisfies $|b(j)|\le 1$ and \eqref{eq:embeddeda} for all $j$, where $\theta(j)$ is determined by
     \begin{equation*}
       \cos{\theta(j)}=\frac{1-{c(j)}^2}{1+{c(j)}^2},\;\;\sin{\theta(j)}=\frac{-2c(j)}{1+{c(j)}^2}.
     \end{equation*}
    \end{theorem}

\begin{theorem} \label{theorem:hyp}
 For $(\x,\n)$, $c(j)$ like in Theorem \ref{theorem:trig}, and assuming $c(j) \neq \pm 1$,
 $K \equiv -1$ if and only if $f(j)$ and $b(j)$ satisfy 
 \begin{equation}\label{eq:solhyp}
  \left\{ \,
  \begin{aligned}
   f(j)=A\cdot\prod_{s=0}^{j-1}\frac{1-c(s)}{1+c(s)}+B\cdot\prod_{s=0}^{j-1}\frac{1+c(s)}{1-c(s)},  \\
   b(j)=-A\cdot\prod_{s=0}^{j-1}\frac{1-c(s)}{1+c(s)}+B\cdot\prod_{s=0}^{j-1}\frac{1+c(s)}{1-c(s)}
  \end{aligned}
  \right.
 \end{equation}
 for some $A$, $B\in \Rr$ and $c(j)$ such that \eqref{eq:solhyp} satisfies $|b(j)|\le 1$ and \eqref{eq:embeddeda} for all $j$. 
 In particular, when $|c(j)|<1$, we have the following three cases:
 \begin{enumerate}
 \item When $\kappa=1$, or equivalently $f(j)^2-b(j)^2=0$, then $A=0$ or $B=0$.
 \item When $\kappa<1$, or $f(j)^2-b(j)^2>0$, then $AB>0$ (``$+$" and ``$-$" below correspond to $A, B>0$ and $A, B<0$) and 
 \begin{equation}
  \left\{ \,
  \begin{aligned}
   f(j)=\pm 2\sqrt{AB}\bigg(\cosh{\Big(\log{\sqrt{\frac{B}{A}}}+{\sum_{s=0}^{j-1}\log{
   \frac{1+c(s)}{1-c(s)}}}}\Big)\bigg),  \\
   b(j)=\pm 2\sqrt{AB}\bigg(\sinh{\Big(\log{\sqrt{\frac{B}{A}}}+{\sum_{s=0}^{j-1}\log{
   \frac{1+c(s)}{1-c(s)}}}}\Big)\bigg).
  \end{aligned}
  \right.
 \end{equation}
 \item When $\kappa>1$, or $f(j)^2-b(j)^2<0$, then $AB<0$ (``$+$" and ``$-$" below correspond to $A<0$ and $A>0$) and
 \begin{equation}
  \left\{ \,
  \begin{aligned}
   f(j)=\pm 2\sqrt{|AB|}\bigg(\sinh{\Big(\log{\sqrt{\frac{|B|}{|A|}}}+{\sum_{s=0}^{j-1}\log{
   \frac{1+c(s)}{1-c(s)}}}}\Big)\bigg),  \\
   b(j)=\pm 2\sqrt{|AB|}\bigg(\cosh{\Big(\log{\sqrt{\frac{|B|}{|A|}}}+{\sum_{s=0}^{j-1}\log{
   \frac{1+c(s)}{1-c(s)}}}}\Big)\bigg).
  \end{aligned}
  \right.
 \end{equation}
 \end{enumerate}
\end{theorem}

\begin{proof} [Proof of Theorems \ref{theorem:trig} and \ref{theorem:hyp}]
 From \eqref{eq:gauss}, when $K=1$, we need to solve
 \begin{equation} \label{eq:trig}
   \left\{ \,
   \begin{aligned}
    &-c(j)\sigma f(j)=\delta b(j),\\
    &c(j)\sigma b(j)=\delta f(j).
   \end{aligned}
   \right.
 \end{equation}
 Then \eqref{eq:trig} becomes
 \begin{equation} \label{eq:trig2}
   \left\{ \,
   \begin{aligned}
    &f(j+1)=\frac{1-{c(j)}^2}{1+{c(j)}^2}f(j)+\frac{2c(j)}{1+{c(j)}^2}b(j),\\
    &b(j+1)=\frac{-2c(j)}{1+{c(j)}^2}f(j)+\frac{1-{c(j)}^2}{1+{c(j)}^2}b(j).
   \end{aligned}
   \right.
 \end{equation}
 Then we can find that the general solution of \eqref{eq:trig2} is
 \eqref{eq:soltrig} for $A, B\in \Rr$. 
   
 Next we consider the case $K=-1$. Then we need to solve
 \begin{equation}\label{eq:hyp}
  \left\{ \,
  \begin{aligned}
    &c(j)\sigma f(j)=\delta b(j),\\
    &c(j)\sigma b(j)=\delta f(j).
  \end{aligned}
  \right.
 \end{equation}
 By $c(j)\neq \pm 1$, \eqref{eq:hyp} can be written as
 \begin{equation} \label{eq:hyp2}
  \left\{ \,
  \begin{aligned}
    &f(j+1)=\frac{1+{c(j)}^2}{1-{c(j)}^2}f(j)+\frac{2c(j)}{1-{c(j)}^2}b(j),\\
    &b(j+1)=\frac{2c(j)}{1-{c(j)}^2}f(j)+\frac{1+{c(j)}^2}{1-{c(j)}^2}b(j).
  \end{aligned}
  \right.
 \end{equation}
 The general solution of \eqref{eq:hyp2} is as in \eqref{eq:solhyp},
 and the three cases after that follow by direct computation.
\end{proof}

\begin{remark}
 Note that $a(j)$ is determined up to sign in the above theorems, and 
 then $h(j)$ is determined up to a global translation by \eqref{eq:embedded}.
\end{remark}
  
\begin{remark}
 In the smooth case, when we set $c(u)=\pm\sqrt{{f^{\prime}(u)}^2+{h^{\prime}(u)}^2}$, 
 we have
 \begin{equation*} \label{eq:gausssm}
  K(u)=\frac{-b^{\prime}(u)}{c(u)f(u)},\;\;
  c(u)b(u)=f^{\prime}(u),\;\;
  -c(u)a(u)=h^{\prime}(u),
 \end{equation*}
 and Equations \eqref{eq:edge1} and \eqref{eq:gauss} can be regarded as a discretization
 of this.
\end{remark}

\begin{remark}
 In \cite{P.Z}, Equations \eqref{eq:trig2} and \eqref{eq:hyp2} are called 
 the scalar discrete trigonometric system and the scalar discrete 
 hyperbolic system, respectively.
\end{remark}

\begin{remark}
    We can describe $f$, $h$, $a$, $b$ explicitly in terms of 
    Jacobi elliptic functions satisfying 
    $$\text{sn}(u,\kappa)^2+\text{cn}(u,\kappa)^2=1,\;\;
   \text{dn}(u,\kappa)^2+\kappa^2 \text{sn}(u,\kappa)^2=1.$$
    For example, when $K=1$ and
    $$c(j)=\frac{- \text{sn}\big(\frac{\Theta}{2},\kappa \big) \text{dn}\big(\frac{(2 j + 1)\Theta}{2},\kappa\big)}{\text{cn}\big(\frac{\Theta}{2},\kappa\big)},\;\;A=\kappa,\;\;B=0$$
    for some $\Theta$, $\kappa \in \Rr$, we have
    \begin{equation*} 
      f(j)=\kappa\textnormal{cn}(\Theta j,\kappa)
      ,\;\;a(j)=\textnormal{dn}(\Theta j,\kappa)
      ,\;\;b(j)=\kappa\textnormal{sn}(\Theta j,\kappa),
    \end{equation*}
    and $h$ can be described explicitly as a summation in terms of $f$, $a$, and $b$, using Equation \eqref{eq:embeddedc}. 
   
   When $K=-1$ and
   $$c(j)=- \kappa\text{sn}\Big(\frac{\Theta}{2},\kappa\Big) \text{sn}\Big(\frac{(2 j + 1)\Theta}{2},\kappa\Big),\;\;A=\frac{1-\kappa}{2\kappa},\;\;B=\frac{1+\kappa}{2\kappa},$$
   we have 
    \begin{equation*} 
      f(j)=\frac{1}{\kappa}\textnormal{dn}(\Theta j,\kappa)
      ,\;\;a(j)=\textnormal{sn}(\Theta j,\kappa)
      ,\;\;b(j)=\textnormal{cn}(\Theta j,\kappa),
    \end{equation*}
    and similarly to arguments in \cite{KSU} for the asymptotic net case, we can compute
    \begin{equation*}
      h(j)=\kappa\Big(\int_0^{\Theta j} \text{sn}(u,\kappa)^2du-2 j \int_0^{\Theta/2} \text{sn}(u,\kappa)^2du\Big),
    \end{equation*}
    This is useful for understanding the singular behavior of these 
    discrete surfaces of revolution. (For singularities in the case of smooth surfaces, see \cite{UYS}, for example.) 
    For a c-net $(\x,\n)$, the principal curvatures $R$ are defined on edges in \cite{BPW,BHR,RY} as
    \begin{equation*} \label{eq:principal curvature}
      \delta_j \n(j,k)=- R_{(j+1,k)(j,k)} \delta_j \x(j,k),\;\;
      \delta_k \n(j,k)=- R_{(j,k+1)(j,k)} \delta_k \x(j,k).
    \end{equation*}
    Moreover, when $(\x,\n)$ has non-zero constant Gaussian curvature,
    a vertex $\x(j,k)$ is singular if (\cite{RY}, see also \cite{YM})
    \begin{equation*} \label{eq:singular}
      R_{(j,k)(j-1,k)}R_{(j+1,k)(j,k)}\leq 0 \;\;\text{or} \;\; 
      R_{(j,k)(j,k-1)}R_{(j,k+1)(j,k)}\leq 0.
    \end{equation*}
    Using these definitions and Jacobi elliptic functions, we see that singular vertices appear, 
    looking like discrete analogs of the cuspidal edges on the corresponding smooth surfaces,
    see Figure \ref{figure:singular}. For the surfaces in that figure, $\Theta$ is determined as follows:
    let 
    $$ K(\kappa):=\int_{0}^{1}\frac{dt}{\sqrt{(1-t^2)(1-\kappa^2t^2)}}$$
    be the complete elliptic integral of the first kind. When $\kappa=1$, we set $\Theta>0$ to be some positive real number, and
    when $0<\kappa<1$, we set $\Theta=\frac{K(\kappa)}{j_0}$, 
    and when $\kappa>1$, we set $\Theta=\frac{K(\frac{1}{\kappa})}{\kappa j_0}$ for some integer $j_0\geq 2$.
\end{remark}

\begin{figure}[htbp] 
  \begin{flushleft}
  \setlength{\tabcolsep}{0pt}
  \renewcommand{\arraystretch}{0}
  \begin{tabular}{ccc}
  \includegraphics[width=5.0cm, trim=2.1cm 0cm 2.1cm 0cm, clip]{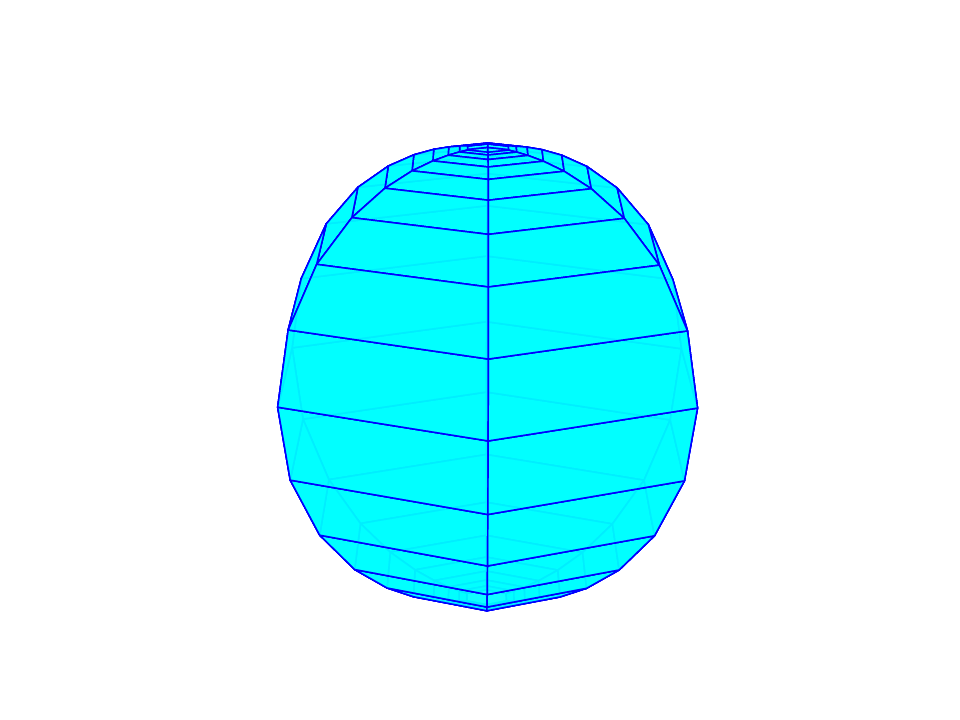} &
  \includegraphics[width=5.0cm, trim=3.5cm 0cm 3.5cm 0cm, clip]{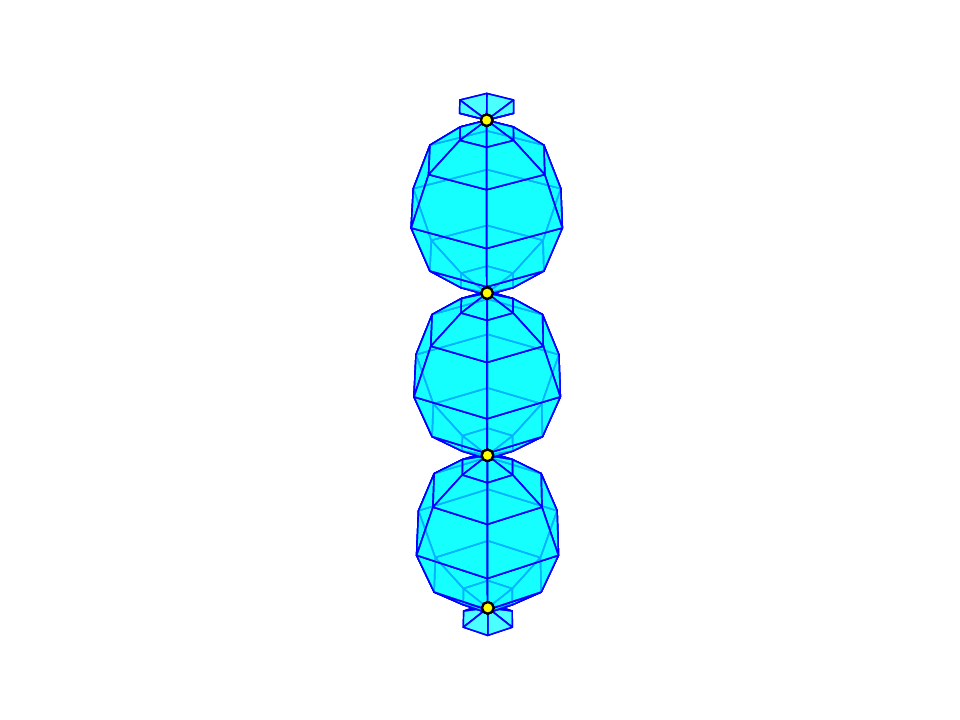} &
  \includegraphics[width=5.0cm, trim=3.5cm 0cm 3.5cm 0cm, clip]{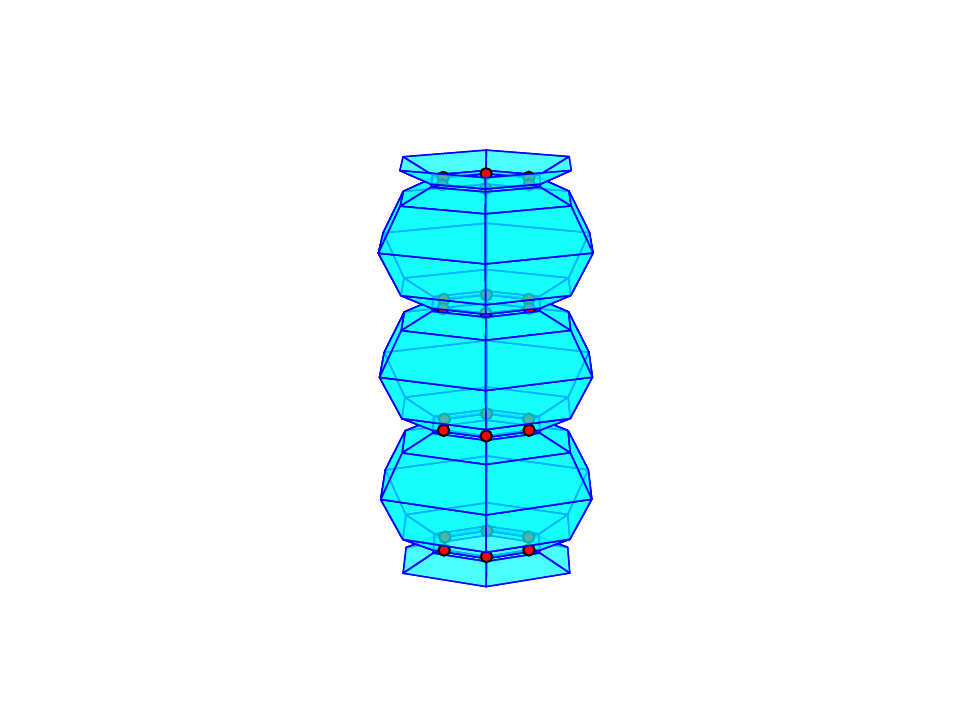} \\
  \noalign{\vspace{-0.5cm}} 
  \includegraphics[width=5.0cm, trim=3.2cm 0cm 3.2cm 0cm, clip]{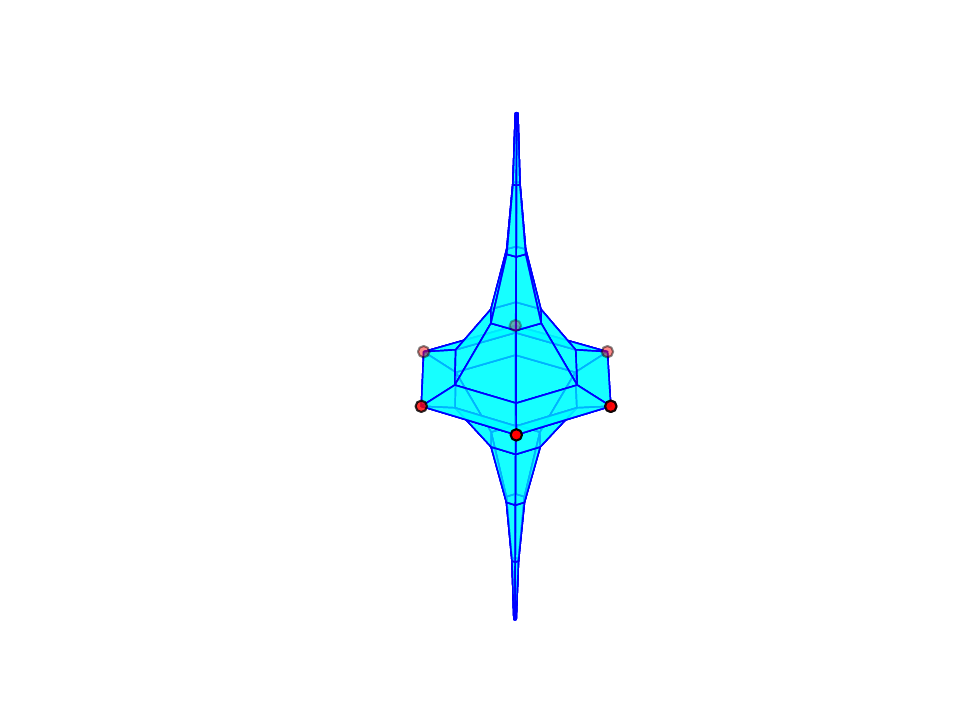} &
  \includegraphics[width=5.0cm, trim=3.2cm 0cm 3.2cm 0cm, clip]{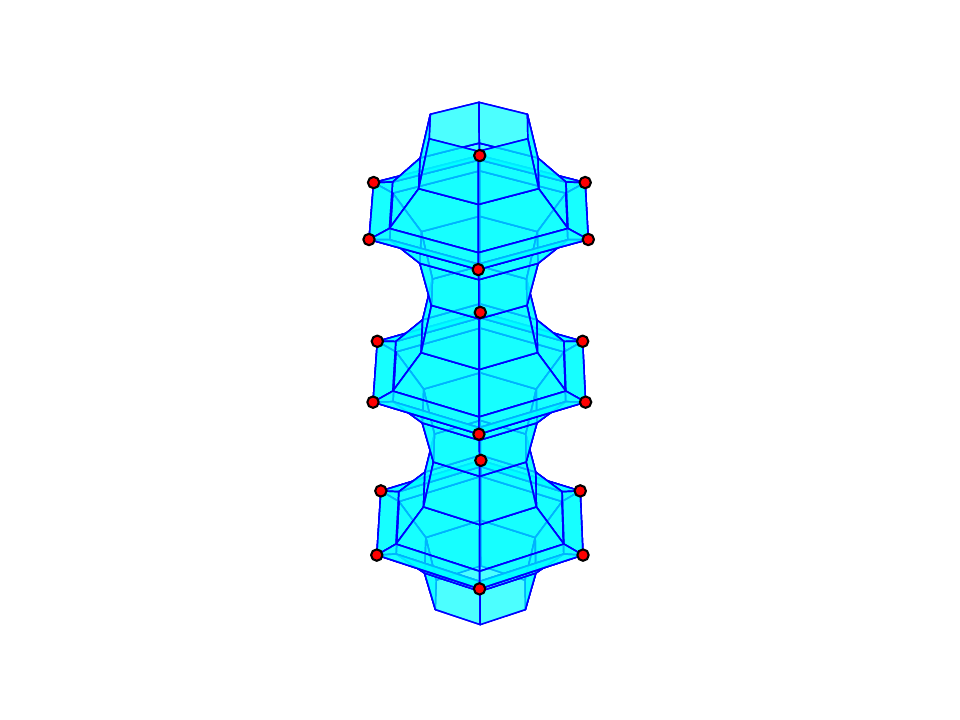} &
  \includegraphics[width=5.0cm, trim=3.0cm 0cm 3.0cm 0cm, clip]{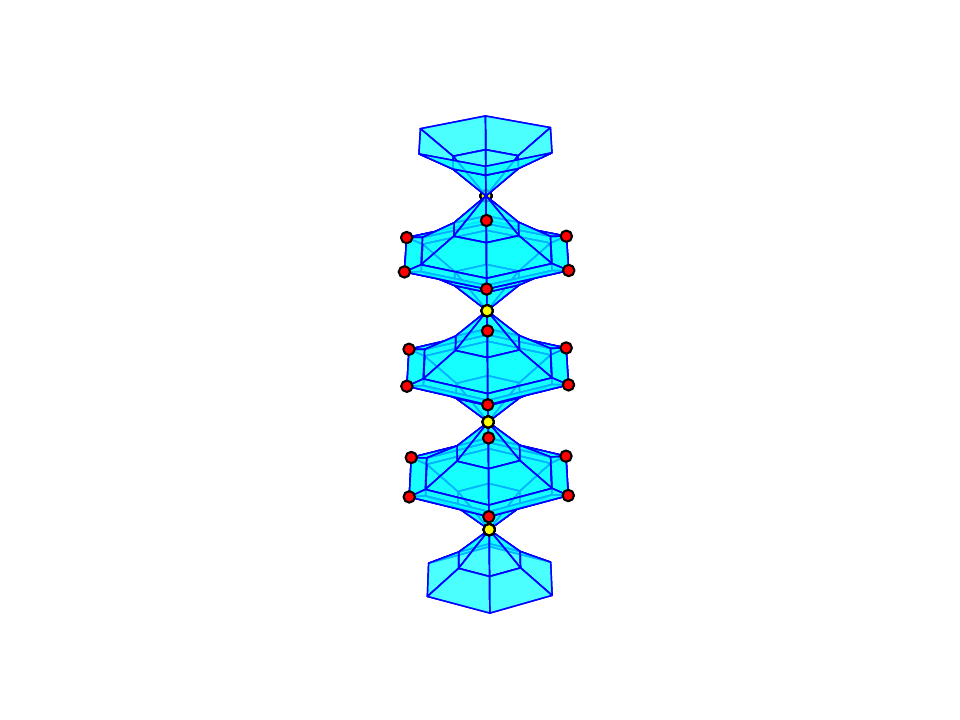}
 \end{tabular}

  \caption{
    Upper row (left–right): The image of $\x$ for rc-nets that have $K\equiv 1$ with degenerate edges and singular vertices, for
    $\kappa=1$, $0<\kappa<1$, $\kappa>1$.
    Lower row (left–right): The image of $\x$ for rc-nets that have $K\equiv -1$ with degenerate edges and singular vertices, again for
    $\kappa=1$, $0< \kappa<1$, $\kappa>1$. 
    Yellow dots indicate degenerate edges in the rotational direction, whereas red dots indicate singular vertices.
  }
  \label{figure:singular}
 \end{flushleft}
\end{figure}

\section{Flat connections for CGC circular nets of revolution}

Let us consider flat connections $\eta$ corresponding to CGC rc-nets $(\x,\n)$ parametrized by 
\eqref{eq:surf and normal}.
The following theorems provide rotationally invariant $\eta$, 
where we say $\eta$ (or a Lax pair $(L,M)$) is \textbf{$k$-invariant}
if, for all $t$ and $(j,k)$, $\eta_{(j+1,k)(j,k)}=\eta_{(j+1,k+1)(j,k+1)}$ and $\eta_{(j,k+1)(j,k)}=\eta_{(j,k+2)(j,k+1)}$, and we can similarly define $j$-invariance.
Given our choices for formulating $\eta$ in Section 1, it is most natural to consider $j$-invariance in the first case of the next theorem, and rather $k$-invariance in the latter two cases.

\begin{theorem} \label{theorem:compatibility}
  We have the following three cases of CGC rc-nets: 
  \begin{enumerate}
  \item For the flat connection $\eta$ in \eqref{eq:connection-CMC},
  if 
  \begin{equation} \label{eq:compa1}
   \begin{aligned}
   & u_{\mathfrak{qp}}\neq - u_{\mathfrak{rs}},\;\;v_{\mathfrak{sp}}=v_{\mathfrak{rq}},\;\;
   v_{\mathfrak{sp}}^2=u_{\mathfrak{qp}}u_{\mathfrak{rs}},\;\;\text{and}\\
   & \hfrak_{\mathfrak{sp}}=\hfrak_{\mathfrak{rq}}=
   \frac{v_{\mathfrak{sp}}(\overline{\vfrak_{\mathfrak{qp}}}-\vfrak_{\mathfrak{rs}})}{i(u_{\mathfrak{qp}}-u_{\mathfrak{rs}})}\in \Rr, \;\;\text{when}\;\;u_{\mathfrak{qp}}\neq u_{\mathfrak{rs}}\;\;\text{or}\\
   & \hfrak_{\mathfrak{sp}}=\hfrak_{\mathfrak{rq}}=\frac{2i(u_{\mathfrak{qp}}v_{\mathfrak{sp}}-u_{\mathfrak{qp}}\inv v_{\mathfrak{sp}}\inv)}{\vfrak_{\mathfrak{qp}}-\vfrak_{\mathfrak{rs}}} \in \Rr,\;\;
   \overline{\vfrak_{\mathfrak{qp}}}=\vfrak_{\mathfrak{rs}},\;\;\vfrak_{\mathfrak{qp}}\neq \vfrak_{\mathfrak{rs}}, \;\;\text{when}\;\;u_{\mathfrak{qp}}= u_{\mathfrak{rs}}
   \end{aligned}
  \end{equation}
  for all faces, and if $\alpha|_{t=0}$ and $\Re \vfrak$ are constants,
  then compatibility \eqref{eq:compatibilitycmc} holds and
  $\eta_{e_2}$ is $j$-invariant. Moreover, when $\eta_{e_1}$ is also $j$-invariant and 
  \begin{equation} \label{eq:angle and embedded1}
    \alpha|_{t=0}=\frac{2}{\sqrt{1-\kappa^2}\sin \frac{\theta}{2}},\;\;
    \Re \vfrak=\frac{2 \cot \frac{\theta}{2}}{\sqrt{1-\kappa^2}},
    \;\;(u_{\mathfrak{qp}}-u_{\mathfrak{qp}}\inv)(u_{\mathfrak{rs}}-u_{\mathfrak{rs}}\inv) >0
  \end{equation}
  for $\theta \in (-\pi,0) \cup (0,\pi)$ and $0<\kappa<1$, then $\eta$ produces a $K \equiv 1$ rc-net as in \eqref{eq:surf and normal} satisfying \eqref{eq:pgauss}, switching the roles of $j$ and $k$, via \eqref{eq:sym-bob} with $t=0$ and $\xi+2=\tau=0$, and 
  every $K\equiv 1$ rc-net with $0<\kappa<1$ and $\text{sign}(a(j+1))=\text{sign}(a(j))$ can be produced this way, up to rigid motion.

  \item For the flat connection $\eta$ in \eqref{eq:connection-CMC},
  if 
  \begin{equation} \label{eq:compa2}
    \begin{aligned}
    & v_{\mathfrak{sp}}\neq - v_{\mathfrak{rq}},\;\;u_{\mathfrak{qp}}=u_{\mathfrak{rs}},\;\;
    u_{\mathfrak{qp}}^2=v_{\mathfrak{sp}}v_{\mathfrak{rq}},\;\;\text{and}\\
    & \vfrak_{\mathfrak{qp}}=\vfrak_{\mathfrak{rs}}=\frac{i u_{\mathfrak{qp}}(\overline{\hfrak_{\mathfrak{sp}}}-\hfrak_{\mathfrak{rq}})}{v_{\mathfrak{sp}}-v_{\mathfrak{rq}}}\in \Rr,\;\;\text{when}\;\;v_{\mathfrak{sp}}\neq v_{\mathfrak{rq}}\;\;\text{or}\\
    & \vfrak_{\mathfrak{qp}}=\vfrak_{\mathfrak{rs}}=\frac{2i(u_{\mathfrak{qp}}v_{\mathfrak{sp}}-u_{\mathfrak{qp}}\inv v_{\mathfrak{sp}}\inv)}{\hfrak_{\mathfrak{rq}}-\hfrak_{\mathfrak{sp}}} \in \Rr,\;\;\overline{\hfrak_{\mathfrak{rq}}}=\hfrak_{\mathfrak{sp}},\;\;\hfrak_{\mathfrak{rq}}\neq \hfrak_{\mathfrak{sp}},\;\;\text{when}\;\;v_{\mathfrak{sp}}= v_{\mathfrak{rq}}
    \end{aligned}
  \end{equation}
  for all faces, and if $\beta|_{t=0}$ and $\Re \hfrak$ are constants,
  then compatibility \eqref{eq:compatibilitycmc} holds and 
  $\eta_{e_1}$ is $k$-invariant. Moreover, when $\eta_{e_2}$ is also $k$-invariant and
  \begin{equation} \label{eq:angle and embedded2}
    \beta|_{t=0}=\frac{2}{\sqrt{\kappa^2-1}\sin \frac{\theta}{2}},\;\;\Re \hfrak=\frac{2 \cot \frac{\theta}{2}}{\sqrt{\kappa^2-1}},\;\;
    u \neq \pm 1,\;\;v_{\mathfrak{sp}} v_{\mathfrak{rq}} \neq 0
  \end{equation}
  for $\theta \in (-\pi,0) \cup (0,\pi)$ and $\kappa>1$, then $\eta$ produces a $K \equiv 1$ rc-net as in \eqref{eq:surf and normal} satisfying \eqref{eq:pgauss} via \eqref{eq:sym-bob} with $t=0$ and $\xi+2=\tau=0$, and 
  every $K\equiv 1$ rc-net with $\kappa>1$ can be produced this way, up to rigid motion.

  \item For the flat connection $\eta$ in \eqref{eq:connection-CGC},
  if 
  \begin{equation} \label{eq:compa3}
    \begin{aligned}
    & v_{\mathfrak{sp}}\neq v_{\mathfrak{rq}},\;\;u_{\mathfrak{qp}}=u_{\mathfrak{rs}},\;\;
    u_{\mathfrak{qp}}^2=v_{\mathfrak{sp}}v_{\mathfrak{rq}},\;\;\text{and}\\
    & \vfrak_{\mathfrak{qp}}=\vfrak_{\mathfrak{rs}}=\frac{u_{\mathfrak{qp}}(-\overline{\hfrak_{\mathfrak{sp}}}+\hfrak_{\mathfrak{rq}})}{v_{\mathfrak{sp}}+v_{\mathfrak{rq}}} \in i\Rr,\;\;\text{when}\;\;v_{\mathfrak{sp}}\neq -v_{\mathfrak{rq}}\;\;\text{or}\\
    & \vfrak_{\mathfrak{qp}}=\vfrak_{\mathfrak{rs}}=\frac{2(u_{\mathfrak{qp}}v_{\mathfrak{sp}}+u_{\mathfrak{qp}}\inv v_{\mathfrak{sp}}\inv)}{\hfrak_{\mathfrak{rq}}-\hfrak_{\mathfrak{sp}}}\in i\Rr,\;\;\overline{\hfrak_{\mathfrak{rq}}}=\hfrak_{\mathfrak{sp}},\;\;\hfrak_{\mathfrak{rq}}\neq \hfrak_{\mathfrak{sp}},\;\;\text{when}\;\;v_{\mathfrak{sp}}= -v_{\mathfrak{rq}}
    \end{aligned}
  \end{equation}
  for all faces, and if $\beta|_{t=0}$ and $\Re \hfrak$ are constants,
  then compatibility \eqref{eq:compatibilitycgc} holds and 
  $\eta_{e_1}$ is $k$-invariant. Moreover, when $\eta_{e_2}$ is also $k$-invariant and
  \begin{equation} \label{eq:angle and embedded3}
    \beta|_{t=0}=\frac{2 \kappa}{\sin \frac{\theta}{2}},\;\;\Re \hfrak=2 \kappa\cot \frac{\theta}{2},\;\; 
    \Re u \neq 0,\;\;(\Re v_{\mathfrak{sp}}) (\Re v_{\mathfrak{rq}}) <0 
  \end{equation}
  for $\theta \in (-\pi,0) \cup (0,\pi)$ and $\kappa>0$, then $\eta$ produces a $K\equiv-1$ rc-net as in \eqref{eq:surf and normal} satisfying \eqref{eq:ngauss} via \eqref{eq:sym-bob} with $t=0$ and $\xi-2=\tau=0$, and 
  every $K\equiv -1$ rc-net can be produced this way, up to rigid motion.
\end{enumerate}
\end{theorem}

Before proving this theorem, we present two lemmas and a proposition about rotationally invariant flat connections.

\begin{lemma} \label{lemma:eigen}
  Let $\eta$ be a k-invariant flat connection.
  Then there exist a k-invariant map $P:\Zz \to \Hh^{\times}$ and a constant diagonal matrix $D$ such that 
  \begin{equation} \label{eq:diag}
    M_{(j,k+1)(j,k)}=P(j)DP(j)\inv,\;\; L_{(j+1,k)(j,k)}=P(j+1)P(j)\inv.
  \end{equation}
\end{lemma}
\begin{proof}
  Since $M_{(j,k+1)(j,k)}$ is a non-zero quaternion and so a normal matrix, 
  $M_{(j,k+1)(j,k)}$ can be diagonalized by a special unitary matrix.
  Therefore, there exists a quaternion $P(j)$ so that 
  $M_{(j,k+1)(j,k)}=P(j)DP(j)\inv$, where
  $$D=\matrix{\omega_{+}}{0}{0}{\omega_{-}}\in \Hh$$
  and $\omega_{+},\omega_{-}$ are the eigenvalues of $M_{(j,k+1)(j,k)}$.
  By the compatibility condition and $k$-invariance,
  \begin{equation} \label{eq:k-inv}
  M_{(j+1,k+1)(j+1,k)}=L_{(j+1,k)(j,k)} M_{(j,k+1)(j,k)}L_{(j+1,k)(j,k)}\inv,
  \end{equation}
  and then the eigenvalues of $M_{(j,k+1)(j,k)}$ are constant with respect to both $j$ and $k$.
  Namely, $D$ is constant. Moreover, \eqref{eq:k-inv} implies
  the columns of $L_{(j+1,k)(j,k)}P(j)$ are eigenvectors of $M_{(j+1,k+1)(j+1,k)}$ 
  with eigenvalues $\omega_{+}, \omega_{-}$. 
  We can then choose $P$ so that $P(j+1)=L_{(j+1,k)(j,k)}P(j)$.
  \end{proof}

\begin{lemma} \label{lemma:associated}
  Let $\eta$ be a k-invariant flat connection,
  with $D$ and $P$ as in \eqref{eq:diag}.
  Set 
  $$Q=P(0)\inv \Phi(0,0),\;\;
  T=\xi \left[Q^{-1} \partial_t Q\right]^{\tr = 0},\;\;
  \mathbf{y}=\x-T,\;\;
  R=Q\inv D Q.$$
  Then, via \eqref{eq:sym-bob}, for all $t$, $\eta$ produces
  \begin{equation} \label{eq:rotation}
    \mathbf{y}(j,k)=R^{-k}\mathbf{y}(j,0)R^k +\xi k Q\inv \left[D^{-1} \partial_t D\right]^{\tr = 0}Q,\;\;\;
    \n(j,k)=R^{-k}\n(j,0)R^k.
  \end{equation}
  Here $R^{-k} \mathbf{y} R^{k}$ represents rotation of $ \mathbf{y}$ about
  axis $\{aQ\inv (-i\boldsymbol{\sigma}_3) Q | a \in \Rr\}$, and the vectors
  $Q\inv \left[D^{-1} \partial_t D\right]^{\tr = 0} Q$
  and $Q\inv (-i\boldsymbol{\sigma}_3) Q$ are parallel.
\end{lemma}
\begin{proof}
  From Lemma \ref{lemma:eigen},
  \begin{equation} \label{eq:solution}
    \Phi(j,k)=P(j)D^k P(0)\inv \Phi(0,0)
  \end{equation}
  is a parallel frame with an initial value $\Phi(0,0)$, and
  the Sym formula gives
  \begin{equation*}
    \begin{aligned}
     \n(j,k)&=\Phi(j,k)\inv (-i \boldsymbol{\sigma}_3)\Phi(j,k)\\
     &=Q\inv D^{-k} Q \Phi(j,0)\inv (-i \boldsymbol{\sigma}_3)\Phi(j,0) Q\inv D^k Q\\
     &=R^{-k} \n(j,0) R^k,
    \end{aligned}
  \end{equation*}
  \begin{equation*}
    \begin{aligned}
    \mathbf{y}(j,k)&=\xi \left[\Phi(j,k)\inv \partial_t \Phi(j,k)\right]^{\tr = 0} + \tau \n(j,k) - T\\
    &=R^{-k} \mathbf{y}(j,0)R^k
    +\xi Q\inv \left[D^{-k} \partial_t (D^k)\right]^{\tr = 0}Q\\
    &=R^{-k} \mathbf{y}(j,0)R^k +\xi k Q\inv \left[D^{-1} \partial_t D\right]^{\tr = 0}Q.
    \end{aligned}
  \end{equation*}
  As $D$ is diagonal,
  $R\inv Q\inv (-i\boldsymbol{\sigma}_3)QR=Q\inv (-i\boldsymbol{\sigma}_3)Q$, and
  $\norm{R\inv  \mathbf{y}(j,0)R}=\norm{ \mathbf{y}(j,0)}$, so
  $$ \mathbf{y}(j,0) \mapsto R^{-k} \mathbf{y}(j,0) R^k$$ 
  represents a rotation with 
  axis through the origin and parallel to $Q\inv (-i\boldsymbol{\sigma}_3) Q$.
  Diagonality of $\left[D^{-1} \partial_t D\right]^{\tr = 0}$ implies 
  $Q\inv \left[D^{-1} \partial_t D\right]^{\tr = 0}Q$ is parallel to $Q\inv(-i\boldsymbol{\sigma_3})Q$.
\end{proof}

\begin{proposition} \label{proposition:closing}
 Let $(\x,\n)$ be a $K\equiv 1$ or $K\equiv -1$ c-net with unitary
 flat connection $\eta$ that produces it at $t=0$, and we assume that $\x$ is locally embedded.
 If $\eta$ is $k$-invariant or $j$-invariant and $\tr M_{(j,1)(j,0)}=2 \cos (\theta/2)$, resp. $\tr L_{(1,k)(0,k)}=2 \cos (\theta/2)$, for $\theta \in  (-\pi,0) \cup (0,\pi)$, 
 then $(\x,\n)$ is an rc-net.
\end{proposition}
\begin{proof}
 Assume that $\eta$ is $k$-invariant.
 Since $\det (M_{(j,k+1)(j,k)})$ is identically $1$ and 
 $$\partial_t (\tr (M_{(j,k+1)(j,k)}))|_{t=0}=0,$$
 we have that the eigenvalues $\omega_{\pm}$ of $M_{(j,k+1)(j,k)}$ satisfy $\partial_t \omega_{\pm} |_{t=0}=0$
 if $\omega_{+} \neq \omega_{-}$, so the $D$ in \eqref{eq:diag} satisfies $\partial_t D|_{t=0}=0$. 
 Noting that $\tr M_{(j,1)(j,0)}=2 \cos (\theta/2)$ implies $\omega_{\pm}=e^{\pm i\theta/2}$ 
 for $\theta \in (-\pi,0) \cup (0,\pi)$, in fact $\omega_{+} \neq \omega_{-}$.
 Therefore, at $t=0$, \eqref{eq:rotation} becomes
 \begin{equation} \label{eq:rotate}
 \mathbf{y}(j,k)=R^{-k}\mathbf{y}(j,0)R^k,\;\; \n(j,k)=R^{-k}\n(j,0)R^k.
 \end{equation}
 Then, by Proposition \ref{proposition:rotational surface}, $(\x,\n)$ is an rc-net.
 The same can be shown when we assume that $\eta$ is 
 $j$-invariant as well.
\end{proof}

\begin{proof}[Proof of Theorem \ref{theorem:compatibility}]
The first sentences of each of the three cases can be checked by direct computations.
Furthermore, in each case, we can see that the assumptions of Proposition \ref{proposition:closing}, namely the $k$-invariance or $j$-invariance, the local embeddedness, and $\tr M_{(j,1)(j,0)}=2 \cos (\theta/2)$, resp. $\tr L_{(1,k)(0,k)}=2 \cos (\theta/2)$, are satisfied, so each $\eta$ produces an rc-net. 
Next, we consider the correspondence with $\kappa$. For an rc-net as in \eqref{eq:surf and normal}, we have
\begin{equation*}
  \frac{||\delta_k \x(j,k)||^2}{4}\pm\frac{(\scal{\delta_k \x(j,k),\n(j,k)})^2}{||\delta_k \x(j,k)||^2}=(f(j)^2\pm a(j)^2)\sin^2 \frac{\theta}{2}.
\end{equation*}
On the other hand, the rc-net produced in cases (1), (2), (3), resp., satisfy
\begin{equation*}
  -\frac{||\delta_j \x(j,k)||^2}{4}+\frac{(\scal{\delta_j \x(j,k),\n(j,k)})^2}{||\delta_j \x(j,k)||^2}=\frac{4}{\alpha^2|_{t=0}}=(1-\kappa^2)\sin^2 \frac{\theta}{2},
\end{equation*}
\begin{equation*}
  \frac{||\delta_k \x(j,k)||^2}{4}-\frac{(\scal{\delta_k \x(j,k),\n(j,k)})^2}{||\delta_k \x(j,k)||^2}=\frac{4}{\beta^2|_{t=0}}=(\kappa^2-1)\sin^2 \frac{\theta}{2},
\end{equation*}
\begin{equation*}
  \frac{||\delta_k \x(j,k)||^2}{4}+\frac{(\scal{\delta_k \x(j,k),\n(j,k)})^2}{||\delta_k \x(j,k)||^2}=\frac{4}{\beta^2|_{t=0}}=\frac{1}{\kappa^2}\sin^2 \frac{\theta}{2}.
\end{equation*}
Therefore, each rc-net produced in each case has the corresponding 
$\kappa$ from Lemma \ref{lemma:gauss}.

Finally, we show that
$\eta$ produces all $K\equiv \pm 1$ rc-nets with the corresponding $\kappa$.
For all $K\equiv \pm 1$ rc-nets as in \eqref{eq:surf and normal}, once 
$a(j)$ and $b(j)$ are given, $K\equiv \pm 1$, Lemma \ref{lemma:gauss}, and \eqref{eq:embedded} imply that 
$f$ (up to sign) and $\delta h$ are determined.
By use of $180^\circ$
rotation about the rotation axis and translation along that axis,
we can remove the sign and translational ambiguities of $f(0)$ and $h$.
Here, we present details for case (3).
For $a(j)$, $b(j)$ satisfying \eqref{eq:embedded}, we set
$$\Im v_{(j,1)(j,0)}=(-1)^j \kappa a(j),\;\;\Im \hfrak_{(j,1)(j,0)}=2 \kappa b(j).$$
Then, since we have \eqref{eq:angle and embedded3} and $v$ is unitary, $\Re v_{(j,1)(j,0)}$ is determined up to 
single choice of sign.
Thus we choose $u_{(j+1,0)(j,0)}$ so that $u_{(j+1,0)(j,0)}^2=v_{(j,1)(j,0)}v_{(j+1,1)(j+1,0)}$,
and $\vfrak$ as in \eqref{eq:compa3}. 
Since $\delta f(j) = 0$ and $\sigma a(j) = 0$ can not simultaneously hold, it follows that $\Re u \neq 0$.
Moreover, for $P(j)$ as in \eqref{eq:diag}, we choose  
$$
\Phi(0,0)|_{t=0}=P(0)|_{t=0}=
\begin{cases}
\frac{i}{\sqrt{2(b(0)+1)}}\matrix{b(0)+1}{a(0)}{a(0)}{-(b(0)+1)} & (b(0)\neq -1),\\
\matrix{0}{i}{i}{0} & (b(0)=-1).
\end{cases}
$$
Since $\Phi(j,0)|_{t=0}=P(j)|_{t=0}\in \SU_2$ and 
the columns of $\Phi(j,0)|_{t=0}$ are eigenvectors of $M_{(j,1)(j,0)}|_{t=0}$ by Lemma \ref{lemma:eigen},
for some unitary function $p(j)$ with $p(0)=i$,
$$
\Phi(j,0)|_{t=0}=P(j)|_{t=0}=
\begin{cases}
\frac{1}{\sqrt{2(b(j)+1)}}\matrix{(b(j)+1)p(j)}{(-1)^{j+1} a(j)p(j)\inv}{(-1)^{j}a(j)p(j)}{(b(j)+1)p(j)\inv} & (b(j)\neq -1),\\
\matrix{0}{-p(j)\inv}{p(j)}{0} & (b(j)=-1).
\end{cases}
$$
Because every entry of $L_{(j+1,0)(j,0)}|_{t=0}$ is pure imaginary, 
$p(j) =\pm 1$ for $j$ odd, and $p(j) =\pm i$ for $j$ even.
Then we can check 
$$\n(j,0)=(P(j)\inv (-i \boldsymbol{\sigma}_3) P(j))|_{t=0}=-i(b(j)\boldsymbol{\sigma}_3+a(j)\boldsymbol{\sigma}_1).$$
Therefore, for any $a$ and $b$ satisfying \eqref{eq:embedded}, 
we see that they indeed correspond to the parameters of a $K\equiv -1$ rc-net $(\x,\n)$. Thus for case (3), we proved $\eta$ produces all $K\equiv -1$ rc-nets with the corresponding $\kappa$. 
For case (1), (2), resp., we choose 
$$u_{(1,k)(0,k)}+u_{(1,k)(0,k)}\inv=\frac{2 a(k)}{\sqrt{1-\kappa^2}},\;\;\Im \vfrak_{(1,k)(0,k)}=\frac{2i b(k)}{\sqrt{1-\kappa^2}},$$ 
$$v_{(j,1)(j,0)}-v_{(j,1)(j,0)}\inv=\frac{2 a(j)}{\sqrt{\kappa^2-1}},\;\;\Im \hfrak_{(j,1)(j,0)}=\frac{2i b(j)}{\sqrt{\kappa^2-1}}.$$ 
\end{proof}

\begin{remark}
  The condition $\text{sign}(a(j+1))=\text{sign}(a(j))$ in case (1) of Theorem \ref{theorem:compatibility} is natural because
  discrete isothermic unduloids appear as $\x \pm \n$. On the other hand, in case (2), discrete isothermic nodoids appear as $\x \pm \n$, where such a condition would not be natural. 
\end{remark}

\begin{remark}
  In cases (1), (2) and (3), $c$ as in \eqref{eq:edge1} satisfies, resp.,
  $$c(k)^2=\frac{(v+v\inv)^2}{|\hfrak|^2},\;\;c(j)^2=\frac{(u-u\inv)^2}{|\vfrak|^2},\;\;c(j)^2=\frac{(2 \Re u)^2}{|\vfrak|^2}.$$
\end{remark}

\begin{remark}
  In the above setting, when allowing $t \neq 0$, and choosing the initial value $\Phi(0,0)=P(0)$ in \eqref{eq:solution}
  for all $t$, we see that the discrete curve $\gamma_j: V \to \Rr^3$ defined by $\gamma_j(\cdot)=\x(j,\cdot)$ 
  via \eqref{eq:rotation} for each $t$ can be parametrized by
  \begin{equation} \label{eq:helix2}
   \gamma_j(k)=(\Upsilon_j \cos (\theta k+\iota_j),\Upsilon_j \sin (\theta k+\iota_j),\mu k+\Psi_j),
  \end{equation}
  where $\theta=\arg \big(\frac{\omega_{+}}{\omega_{-}}\big)$ and $\mu=\frac{1}{2}\xi i(\frac{\partial_t \omega_{+}}{\omega_{+}}-\frac{\partial_t \omega_{-}}{\omega_{-}})$, and
  $\Upsilon_j,\Psi_j,\iota_j$ are $k$-invariant. Thus, we see that the surfaces in the associated families contain discrete helices,
  see Figure \ref{figure:associated}.
\end{remark}

\begin{figure}[htbp] 
  \centering
  \setlength{\tabcolsep}{0pt}
  \renewcommand{\arraystretch}{0}
  \begin{tabular}{ccc}
  \includegraphics[width=7.0cm, angle=90, trim=3cm 1cm 3cm 1cm, clip]{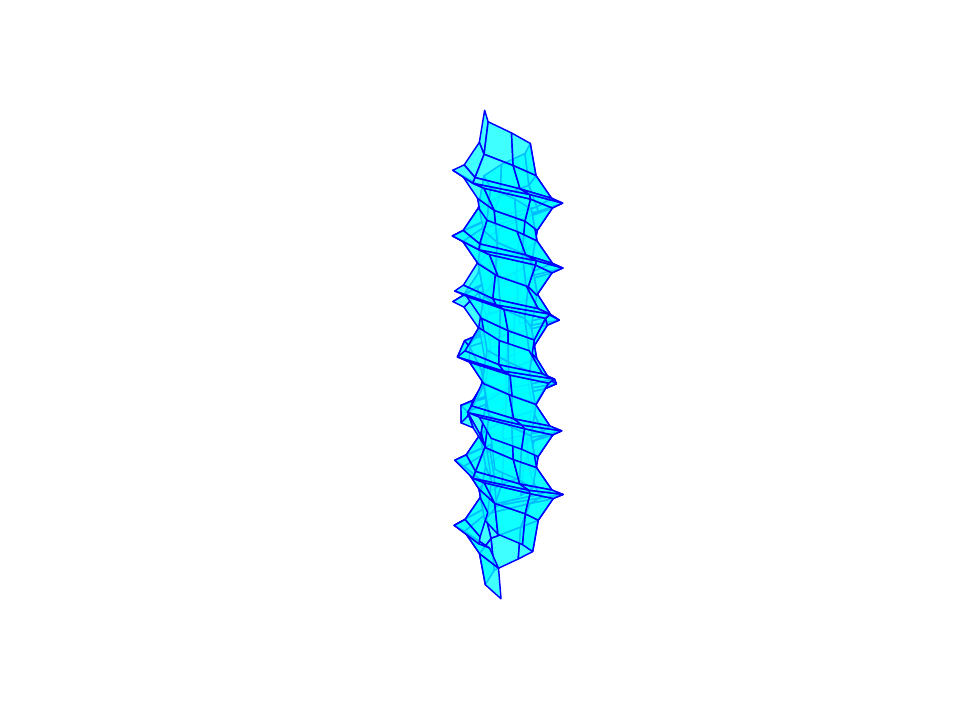} &
  \includegraphics[width=7.0cm, angle=90, trim=3cm 1cm 3cm 1cm, clip]{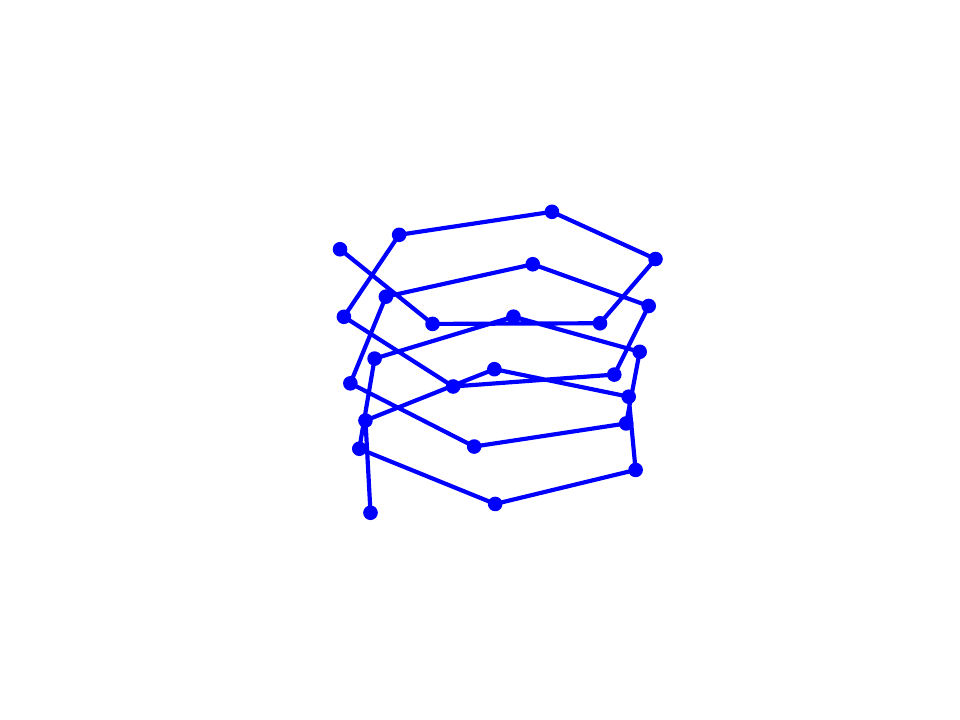} 
 \end{tabular}
 \vspace{-1.5cm}
  \caption{Left: The image of $\x$ for a member of the associated family of a $K\equiv -1$ rc-net (the vertical axis is now pointing rightward).
  Right: The discrete helix $\gamma_{0}$ at $j=0$ in the leftward surface.}
  \label{figure:associated}
\end{figure}

\section{Periodic B\"acklund transforms for cK-nets}

In this section, we construct negative CGC circular nets 
with periodicity, thus annular, by applying the Bäcklund transformation considered in
\cite{HS} to rotationally symmetric discrete surfaces.
In the previous sections, we described the Lax pair in terms of edge-defined functions $u$ and $v$, 
and this was convenient for comparing the $K\equiv1$ and $K\equiv-1$ cases 
seen in \eqref{eq:connection-CMC} and \eqref{eq:connection-CGC} and for understanding the connections as edge-based objects. 
For the purpose of the present section, rather, it is helpful to return to the vertex-defined function $s$ found in the Lax pair $(\mathcal{L},\mathcal{M})$ in \eqref{eq:Lax} originally studied in \cite{HS}. 

Taking the net to be rotational in the $k$-direction with associated $k$-invariant connection having $k$-invariant functions $u$ and $v$, we can choose $s$ to be $s(j)=v_{(j,1)(j,0)}$ with 
$\ell$, $m$, and $\delta_{(1)}$ and $\delta_{(2)}$ as defined in \eqref{eq:delta}.  
Thus the now non-unitary Lax pair $(\mathcal{L},\mathcal{M})$ in \eqref{eq:Lax}  becomes $k$-invariant.

In Proposition \ref{proposition:closing}, we saw a rotation angle $\theta$ that will determine whether the net closes, 
that is, whether 
\begin{equation}\tag{closing}
  \x(j,k+k_0)=\x(j,k),\;\;\n(j,k+k_0)=\n(j,k)
\end{equation}
for some integer $k_0 > 2$, 
for the case of $k$-invariant unitary flat connections.
Although the Lax pair \eqref{eq:Lax} is not unitary, we have a similar closing 
condition here as well. In this case, Equation \eqref{eq:rotate} still holds, and if $\mathcal{M}_{(j,k+1)(j,k)}^{k_0}|_{t=0}$
is a scalar multiple of the identity $I_2$, then $R^{k_0}$ is as well. Applying the Sym formula \eqref{eq:sym-bob}, we can conclude that 
\begin{equation} \label{eq:close}
	\mathcal{M}_{(j,k+1)(j,k)}^{k_0}|_{t=0}= \pm \sqrt{\det \mathcal{M}_{(j,1)(j,0)}^{k_0}}|_{t=0} I_2
\end{equation}
is a closing condition, since $\partial_t(\det \mathcal{M}_{(j,1)(j,0)})|_{t=0}=0$. In particular, this implies the rotational angle $\theta$ in \eqref{eq:close} satisfies $k_{0} \theta \in 2\pi \Zz$.

\subsection{B\"acklund transforms for cK-nets}
In \cite{HS}, the Bäcklund transformation for negative CGC circular nets 
was introduced by means of a Lax pair. By applying this method to 
rotationally symmetric circular nets, we will obtain new periodic 
circular nets. 

\begin{definition} [\cite{HS}]
Let $(\x,\n)$ and $(\tilde \x,\tilde \n)$ be cK-nets.  
Then $(\tilde \x,\tilde \n)$ is a \textbf{B\"acklund 
transformation} of $(\x,\n)$ if 
\begin{enumerate}
	\item the distance between corresponding points of $\x$ and $\tilde\x$ is constantly $|\sin \alpha|$, i.e. 
	\[ | \x - \tilde\x | = |\sin \alpha| \] 
	for some real value $\alpha \in (-\pi,0) \cup (0,\pi)$ called the \textbf{B\"acklund parameter},  
	\item every edge between corresponding points of $\x$ and $\tilde \x$ is perpendicular to both 
	normals $\n$ and $\tilde \n$ at those points, that is, 
	\[ \n, \tilde \n \perp \tilde \x - \x \; , \]
	\item $\n$ and $\tilde \n$ at corresponding points of $\x$ and $\tilde \x$ make 
	angle $\alpha$ between them, i.e. 
	\[ \langle \n, \tilde \n \rangle = \cos \alpha \; . \]
\end{enumerate}
\end{definition}

The transforms $(\tilde \x,\tilde \n)$ can be constructed by a dressing procedure, i.e. a gauge (see Theorem 8, \cite{HS}). 
We choose the B\"acklund matrices as
\begin{equation} 
 W(j,k;\alpha) = \begin{pmatrix}
	\cot \frac{\alpha}{2} \cdot \frac{\tilde s}{s} & i e^t \\
	i e^t & \cot \frac{\alpha}{2} \cdot \frac{s}{\tilde s}
\end{pmatrix},\;\;
V(j,k;\beta)=\begin{pmatrix}
	1 & \frac{i}{e^t} \tan \frac{\beta}{2}  \cdot \hat{s} s \\
	\frac{i}{e^t} \tan \frac{\beta}{2}  \cdot \frac{1}{\hat{s} s} & 1
\end{pmatrix}
\end{equation}
with B\"acklund parameters $\alpha$, $\beta$, and consider the gauging as $\tilde \Phi=W\Phi$ and $\hat \Phi=V\Phi$.  
We note that $\tilde s$ is the same for $(\tilde \x, \tilde \n)$ as $s$ is for $(\x,\n)$. In the latter gauging with $V$, 
we denote the B\"acklund transform as $(\hat \x,\hat\n)$ with corresponding $\hat s$, to distinguish it from the first gauging.  
These B\"acklund matrices come from the asymptotic K-net matrices defined in \cite{BP}, 
as though the edge connecting the transformed point and 
the original point would become an edge of an asymptotic K-net.
The gauge matrices $W$ and $V$ give the new Lax pairs $(\tilde{\mathcal{L}}, \tilde{\mathcal{M}})$ and $(\hat{\mathcal{L}},\hat{\mathcal{M}})$ again of the form \eqref{eq:Lax}, 
which are determined by \eqref{eq:gauged-connection-formula}, that is, we have
\begin{align*}
	&W(j+1,k;\alpha)\mathcal{L}_{(j+1,k)(j,k)}=\tilde{\mathcal{L}}_{(j+1,k)(j,k)} W(j,k;\alpha),\\
	&W(j,k+1;\alpha)\mathcal{M}_{(j,k+1)(j,k)}=\tilde{\mathcal{M}}_{(j,k+1)(j,k)} W(j,k;\alpha),
\end{align*}
and
\begin{align*}
 &V(j+1,k;\beta)\mathcal{L}_{(j+1,k)(j,k)}=\hat{\mathcal{L}}_{(j+1,k)(j,k)} V(j,k;\beta),\\
 &V(j,k+1;\beta)\mathcal{M}_{(j,k+1)(j,k)}=\hat{\mathcal{M}}_{(j,k+1)(j,k)} V(j,k;\beta).
\end{align*}
From this, $\tilde s$ and $\hat s$ can be determined as solutions 
of rational difference equations written via M\"obius transforms 
$\matrix{a_{11}}{a_{12}}{a_{21}}{a_{22}} \bullet z:=\frac{a_{11}z+a_{12}}{a_{21}z+a_{22}}$,
where the difference equations are 
\begin{subequations} \label{eq:riccatiWW}
\begin{equation} \label{eq:riccatiW}
 \tilde s(j+1,k)=A_{(j+1,k)(j,k)} \bullet \tilde s(j,k),\;\; 
 \tilde s(j,k+1)=B_{(j,k+1)(j,k)} \bullet \tilde s(j,k),
\end{equation}
where 
\begin{equation} \label{eq:A}
A=\matrix{\sin{\alpha}\; \ell^{-1} (u t_{1}^{-1}+u^{-1} t_{1})}{(u\inv-u)\cos{\alpha}-u-u \inv}{ (u-u \inv)\cos{\alpha}-u-u \inv}{\sin{\alpha}\; \ell (u \inv  t_{1}^{-1}+u t_{1})},
\end{equation}
\begin{equation} \label{eq:B}
B=\matrix{\sin{\alpha}\; m^{-1} (v t_{2}^{-1}+v^{-1} t_{2})}{(v\inv-v)\cos{\alpha}-v-v \inv}{ (v-v\inv)\cos{\alpha}-v-v \inv}{\sin{\alpha}\; m (v \inv  t_{2}^{-1}+v t_{2})},
\end{equation}
\end{subequations}
and 
\begin{subequations} \label{eq:riccatiVV}
\begin{equation} \label{eq:riccatiV}
	\hat s(j+1,k)=C_{(j+1,k)(j,k)} \bullet \hat s(j,k),\;\;
	\hat s(j,k+1)=D_{(j,k+1)(j,k)} \bullet \hat s(j,k)
\end{equation}
where
\begin{equation} \label{eq:C}
C=\matrix{\sin{\beta}\; \ell(u\inv t_{1}^{-1}+u t_{1})}{(u-u \inv)\cos{\beta}+u+u \inv}{ (u\inv-u)\cos{\beta}+u+u \inv}{\sin{\beta}\; \ell \inv (u   t_{1}^{-1}+u \inv t_{1})},
\end{equation}
\begin{equation} \label{eq:D}
D=\matrix{\sin{\beta}\; m(v\inv t_{2}^{-1}+v t_{2})}{(v-v \inv)\cos{\beta}+v+v \inv}{ (v\inv-v)\cos{\beta}+v+v \inv}{\sin{\beta}\; m \inv (v  t_{2}^{-1}+v \inv t_{2})}.
\end{equation}
\end{subequations}
Here, $t_{i}:=\tan(\delta_{(i)}/2)$ for $i=1,2$ with consistency implying $\tilde t_{(i)}=t_{(i)}$, and $u$, $v$ are defined by \eqref{eq:s}.
The next lemma states that the functions $\tilde{s}$ and $\hat{s}$ for the transformed Lax pairs are again unitary, like $s$.

\begin{lemma} \label{lemma:stildeshat}
	Let $(L,M)$ be a Lax pair as in \eqref{eq:Lax}, and 
	$\tilde s$ and $\hat s$ solutions for \eqref{eq:riccatiW} and \eqref{eq:riccatiV}, respectively.
	If $\alpha \in  (-\pi,0) \cup (0,\pi)$, resp. $\beta \in  (-\pi,0) \cup (0,\pi)$, 
	and the initial value $\tilde s(0,0)$, resp. $\hat s(0,0)$, is unitary, then $\tilde s$, resp. $\hat s$, is unitary for all $(j,k) \in \Zz^2$.
\end{lemma}
\begin{proof}
 When $\alpha\in  (-\pi,0) \cup (0,\pi)$, $A$ takes the form 
 $\matrix{a_{11}}{a_{12}}{\overline{a_{12}}}{\overline{a_{11}}}$,
 giving a M\"obius transformation preserving unitarity. The cases of $B$, $C$ and $D$ are similar.
\end{proof}

When $\alpha \in (-\pi,0) \cup (0,\pi)$, the new frame $\tilde \Phi=W\Phi$ and Sym formula produce a new cK-net $(\tilde \x,\tilde \n)$ at non-degenerate faces (see Theorem 8, \cite{HS}), 
a B\"acklund transformation of $(\x,\n)$, and likewise for $\hat \Phi$.

\subsection{Periodic B\"acklund transforms for circular nets of revolution}

\begin{lemma}
Supposing $B$ and $D$ are $k$-invariant,
then $\tilde s(j,k+N_0)=\tilde s(j,k)$ for all $(j,k) \in \Zz^2$ and some $N_0>1$ if
\begin{equation} \label{eq:period1}
 B^{N_0}_{(j,k+1)(j,k)}=\pm \sqrt{\det B^{N_0}}I_2.
\end{equation}
Similarly, $\hat s(j,k+N_0)=\hat s(j,k)$ for all $(j,k) \in \Zz^2$ if
\begin{equation} \label{eq:period2}
 D^{N_0}_{(j,k+1)(j,k)}=\pm \sqrt{\det D^{N_0}}I_2.
\end{equation}
\end{lemma}
\begin{proof} 
	This follows from $\tilde s(j,k+N_0)=B^{N_0} \bullet \tilde s(j,k)$, and similarly for $\hat s$.
\end{proof}

\begin{theorem} \label{theorem:periodic}
 For a $K=-1$ rc-net $(\x,\n)$ with $k$-invariant data satisfying the closing condition \eqref{eq:close},
 if $\alpha$, resp. $\beta$, is in $(-\pi,0) \cup (0,\pi)$ and 
 \eqref{eq:period1}, resp. \eqref{eq:period2}, holds, then 
 $(\tilde \x,\tilde \n)$, resp. $(\hat \x, \hat \n)$, is periodic with period $\mathrm{LCM}(k_0,N_0)$, 
 the least common multiple of $k_0$ and $N_0$.
\end{theorem}
\begin{proof}
	It remains only to show $\tilde s$, resp. $\hat s$, is periodic under $k \mapsto k+N_0$, which follows from \eqref{eq:period1} and \eqref{eq:period2}.
	Equation \eqref{eq:close} together with the Sym formula gives the result.
\end{proof}

\begin{figure}[htbp]
  \centering
  \setlength{\tabcolsep}{0pt} 
  \renewcommand{\arraystretch}{0} 
  \begin{tabular}{cccc}
    \includegraphics[width=.26\linewidth, trim=4.5cm 0cm 4.5cm 0cm, clip]{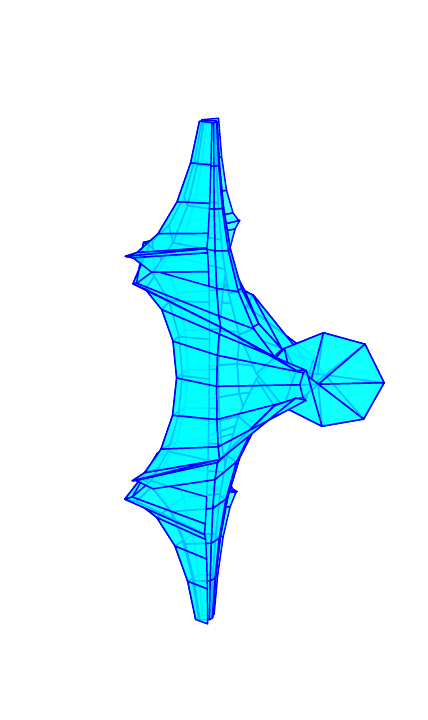} &
    \includegraphics[width=.26\linewidth, trim=5.0cm 0.5cm 5.0cm 0.5cm, clip]{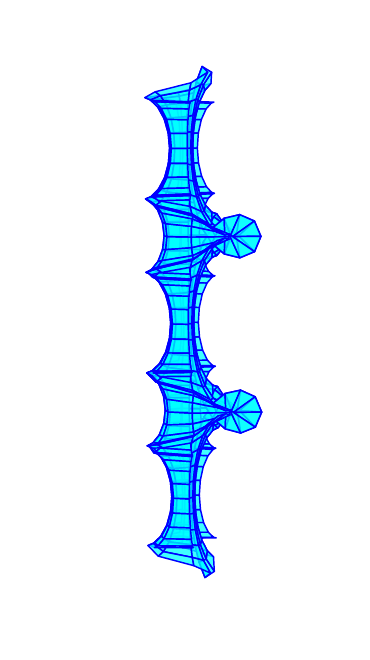} &
    \includegraphics[width=.26\linewidth, trim=4.5cm 0cm 4.5cm 0cm, clip]{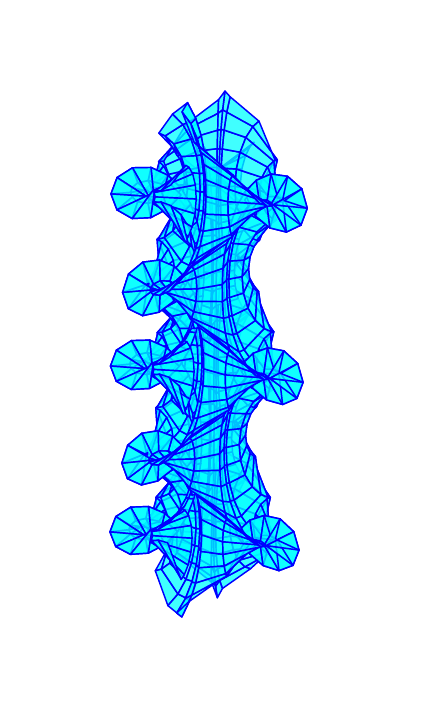} &
    \includegraphics[width=.26\linewidth, trim=4.5cm 0cm 4.5cm 0cm, clip]{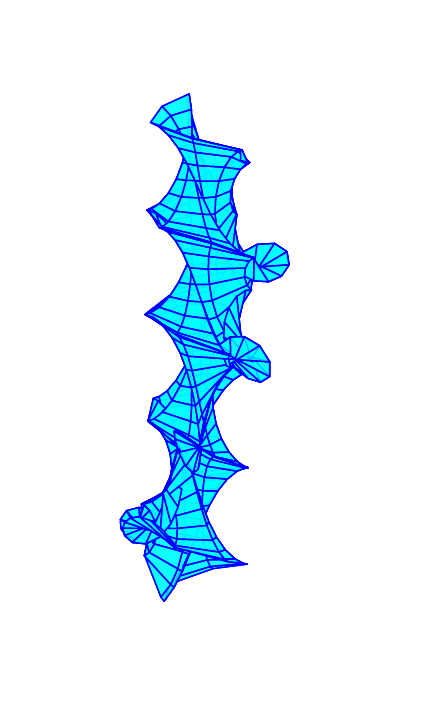}
  \end{tabular}
  \caption{Left: a single non-closed B\"acklund transform with $\alpha=\pi/2$ of a discrete pseudosphere (a discrete Kuen surface). 
	Second from left: a single non-closed B\"acklund transform with $\alpha=\pi/2$ of an rc-net for $0<\kappa<1$. 
	Second from right: a single non-closed B\"acklund transform with $\alpha=\pi/2$ of an rc-net for $\kappa>1$. 
	Right: a single B\"acklund transform of an rc-net for $\kappa>1$, which is periodic under $k \mapsto k+10$, using Theorem \ref{theorem:periodic}.}
\end{figure}

\subsection{Double B\"acklund transforms}
To consider double B\"acklund transformations,
after choosing the first two B\"acklund transformations 
via $W$ and $V$, we require permutability of the second 
two B\"acklund transformations with B\"acklund matrices
\begin{equation*} 
 \hat W(j,k;\alpha) = \matrix{\cot \frac{\alpha}{2} \cdot \frac{\tilde{\hat{s}}}{\hat{s}}}{i e^t}{i e^t}{\cot \frac{\alpha}{2} \cdot \frac{\hat{s}}{\tilde{\hat{s}}}},\;\;
 \tilde V(j,k;\beta)= \matrix{1}{\frac{i}{e^t} \tan \frac{\beta}{2}  \cdot \hat{\tilde{s}} \tilde{s}}{\frac{i}{e^t} \tan \frac{\beta}{2}  \cdot \frac{1}{\hat{\tilde{s}} \tilde{s}}}{1}.
\end{equation*}
We then obtain two frames $\hat{\tilde \Phi}=\tilde V W\Phi$ and $\tilde{\hat \Phi}=\hat W V\Phi$, 
and $\tilde V W=\hat W V$ would imply $\hat{\tilde \Phi}=\tilde{\hat \Phi}$, which is the frame for the desired double B\"acklund transform $(\hat{\tilde \x},\hat{\tilde \n})$ of $(\x,\n)$. 
If the complex parameters $\alpha$ and $\beta$ are chosen so that $(\hat{\tilde \x},\hat{\tilde \n})$ is a net in $\Rr^3$, 
then, as seen in \cite{HS}, $(\hat{\tilde \x},\hat{\tilde \n})$ is again a cK-net (at non-degenerate faces) produced by the Sym formula with $\xi=2$ and $\tau=0$ at $t=0$.

When we assume that $\beta=-\alpha$, then $\hat{\tilde{s}}=\tilde{\hat{s}}$ and we have the permutability formula
\begin{equation} \label{eq:permutability}
\hat{\tilde{s}}=\tilde{\hat{s}}=\frac{1}{s}\cdot\frac{\hat{s}\tilde{s}-\tan^2 \frac{\alpha}{2}}{1-\tan^2 \frac{\alpha}{2}\hat{s}\tilde{s}}.
\end{equation}

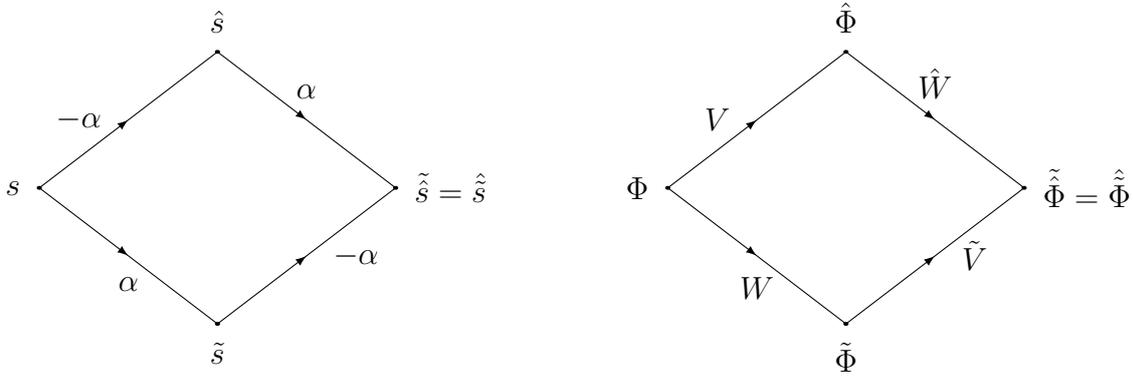
\begin{figure}[htbp]
  \centering
  \begin{minipage}{0.4\textwidth}
    \centering
    \begin{tikzpicture}[x=1.7cm,y=1.7cm,font=\normalsize,
      decoration={markings, mark=at position 0.5 with {\arrow{latex}}}
    ]
      \begin{scope}[scale=0.75]
        \begin{scope}[xscale=1.3]
          \begin{scope}[rotate around={-45:(1,1)}]
            \draw[black,postaction={decorate}] (0,0) -- (2,0)
              node[midway, below=4pt] {$\alpha$};
            \draw[black,postaction={decorate}] (2,0) -- (2,2)
              node[midway, right=6pt] {$-\alpha$};
            \draw[black,postaction={decorate}] (0,0) -- (0,2)
              node[midway, left=6pt] {$-\alpha$};
            \draw[black,postaction={decorate}] (0,2) -- (2,2)
              node[midway, above=4pt] {$\alpha$};

            \fill (0,0) circle[radius=1pt];
            \node[left=3pt] at (0,0) {$s$};

            \fill (2,0) circle[radius=1pt];
            \node[below=3pt] at (2,0) {$\tilde s$}; 

            \fill (2,2) circle[radius=1pt];
            \node[right=3pt] at (2,2) {$\tilde{\hat{s}}=\hat{\tilde{s}}$};

            \fill (0,2) circle[radius=1pt];
            \node[above=3pt] at (0,2) {$\hat s$};    
          \end{scope}
        \end{scope}
      \end{scope}
    \end{tikzpicture}
  \end{minipage}
  \hspace{1.0cm}
  \begin{minipage}{0.5\textwidth}
    \centering
    \begin{tikzpicture}[x=1.7cm,y=1.7cm,font=\normalsize,
      decoration={markings, mark=at position 0.5 with {\arrow{latex}}}
    ]
      \begin{scope}[scale=0.75]
        \begin{scope}[xscale=1.3]
          \begin{scope}[rotate around={-45:(1,1)}]
            \draw[black,postaction={decorate}] (0,0) -- (2,0)
              node[midway, below=4pt] {$W$};
            \draw[black,postaction={decorate}] (2,0) -- (2,2)
              node[midway, right=6pt] {$\tilde V$};
            \draw[black,postaction={decorate}] (0,0) -- (0,2)
              node[midway, left=6pt] {$V$};
            \draw[black,postaction={decorate}] (0,2) -- (2,2)
              node[midway, above=4pt] {$\hat W$};

            \fill (0,0) circle[radius=1pt];
            \node[left=3pt] at (0,0) {$\Phi$};

            \fill (2,0) circle[radius=1pt];
            \node[below=3pt] at (2,0) {$\tilde \Phi$};

            \fill (2,2) circle[radius=1pt];
            \node[right=3pt] at (2,2) {$\tilde{\hat{\Phi}}=\hat{\tilde{\Phi}}$};

            \fill (0,2) circle[radius=1pt];
            \node[above=3pt] at (0,2) {$\hat \Phi$};
          \end{scope}
        \end{scope}
      \end{scope}
    \end{tikzpicture}
  \end{minipage}

  \caption{Left: permutability diagram for creating $\tilde{\hat s}= \hat{\tilde s}$ with $\beta=-\alpha$.
  Right: the corresponding diagram for $\tilde{\hat \Phi}=\hat{\tilde \Phi}$.}
\end{figure}

Here, perhaps $\alpha \notin \Rr$, so $\tilde s$ and $\hat s$ might not be unitary, and we require the next lemma, which utilizes 
\begin{equation*} 
\text{Condition $\mathcal{C}$:}\;\;\;\beta=-\alpha\;\;\text{and}\;\;\sin \alpha \in \Rr.
\end{equation*}

\begin{lemma} \label{lemma:unitary}
 Under Condition $\mathcal{C}$,
 if $|\sin \alpha|\leq 1$, we choose the initial values 
 $\tilde s(0,0)$ and $\hat s(0,0)$ to be unitary, 
 whereas if $|\sin \alpha|>1$, we choose them so that 
 $\tilde s(0,0)= \overline{\hat s(0,0)}$. Then 
 $\hat{\tilde{s}}$ is unitary, even if $\hat{s}$ and 
 $\tilde{s}$ are not.
\end{lemma}
\begin{proof}
  When $|\sin \alpha|\leq 1$, $\sin \alpha \in \Rr$ implies $\alpha$ is real and $\tilde s(j,k)$, $\hat s(j,k)$ are unitary by Lemma \ref{lemma:stildeshat}, and also $\tan^2 \frac{\alpha}{2}\in \Rr$. 
	Thus the lemma holds by Equation \eqref{eq:permutability}.
	When $|\sin \alpha|>1$, $\sin \alpha \in \Rr$ implies
  $\tan^2 \frac{\alpha}{2}$ is unitary. 
  Since $\beta=-\alpha$,  
	we have $C=-\boldsymbol{\sigma}_1 A^{t}\boldsymbol{\sigma}_1$, and $A$ has conjugate diagonal elements and real off-diagonal elements. 
	It follows that 
	if $\tilde s(j,k)=\overline{\hat s(j,k)}$, then $\tilde s(j+1,k)=\overline{\hat s(j+1,k)}$, and similarly $\tilde s(j,k+1)=\overline{\hat s(j,k+1)}$.
	Then, by induction, $\tilde s(j,k)= \overline{\hat s(j,k)}$ for all $(j,k)$ if $\tilde s(0,0)= \overline{\hat s(0,0)}$. Hence again, the lemma holds by \eqref{eq:permutability}.
\end{proof}

In Lemma \ref{lemma:unitary}, when $|\sin \alpha|\leq 1$, then $s$, $\tilde s$, $\hat s$ and $\hat{\tilde{s}}$ are all unitary, so
$W$, $V$, $\hat W$ and $\tilde V$ are all quaternions for all $t$, and the Sym formula implies all the nets 
$(\x,\n)$, $(\tilde \x,\tilde \n)$, $(\hat \x,\hat \n)$ and $(\hat{\tilde{\x}},\hat{\tilde{\n}})$ lie in $\Rr^3 \times \Ss^2$.
However, when $|\sin \alpha| > 1$, $(\tilde \x, \tilde \n)$ and $(\hat \x, \hat \n)$ generally are not constrained to $\Rr^3 \times \Ss^2$, 
but $(\hat{\tilde{\x}},\hat{\tilde{\n}})$ still will be, as in Corollary \ref{lemma:doublebacklund}.

\begin{corollary} \label{lemma:doublebacklund}
  Under the conditions in Lemma \ref{lemma:unitary}, the double B\"acklund transform $(\hat{\tilde{\x}},\hat{\tilde{\n}})$ lies in $\Rr^3 \times \Ss^2$.
\end{corollary}
\begin{proof}
	When $|\sin \alpha|>1$, we have $\tilde s= \overline{\hat s}$, and
	$$|\tilde{s}|^2=\frac{\hat{\tilde{s}} \cot \frac{\alpha}{2} + s \inv \tan \frac{\alpha}{2}}{ s \inv \cot \frac{\alpha}{2} +  \hat{\tilde{s}} \tan \frac{\alpha}{2} }.$$
	Therefore,
	\begin{equation} \label{eq:VW}
	\tilde{V}(j,k;-\alpha)W(j,k;\alpha)=
	\matrix{\tilde{s}( s \inv \cot \frac{\alpha}{2} +  \hat{\tilde{s}} \tan \frac{\alpha}{2} )}
     {i(e^t - e^{-t} s \hat{\tilde{s}})}
     {i(e^t - e^{-t} (s \hat{\tilde{s}})\inv)}
     {\tilde{s}\inv( s \cot \frac{\alpha}{2} +  \hat{\tilde{s}} \inv \tan \frac{\alpha}{2} )}
	\end{equation}
  is a quaternion for all $t$, and $(\hat{\tilde \x},\hat{\tilde \n})$ lies in $\Rr^3 \times \Ss^2$.
\end{proof}

\begin{remark}
In Lemma 7 of \cite{HS}, it was shown that if $\beta=-\alpha$, then $\x$ and $\hat{\tilde \x}$ generate 3D consistent cubes with circular quads on all sides.
\end{remark}

\subsection{Periodic double B\"acklund transforms for circular nets of revolution}
With the conditions as in Lemma \ref{lemma:unitary}, Equation \eqref{eq:VW} shows that the double B\"acklund transforms will be periodic when 
$s$, $\tilde s$, $\hat s$ and $\tilde{\hat s}=\hat{\tilde s}$ are.
This is the content of the following theorem.

\begin{theorem}
 For a $K=-1$ rc-net $(\x,\n)$ with $k$-invariant data satisfying the closing condition \eqref{eq:close}, 
 let $(\hat{\tilde{\x}},\hat{\tilde{\n}})$ be a double B\"acklund transform formed from Condition $\mathcal{C}$ and the conditions in Lemma \ref{lemma:unitary}.
 If either $B$ in \eqref{eq:riccatiWW} satisfies \eqref{eq:period1} or $D$ in \eqref{eq:riccatiVV} satisfies \eqref{eq:period2},
 then $(\hat{\tilde \x},\hat{\tilde \n})$ is periodic with period $\mathrm{LCM}(k_0,N_0)$.
\end{theorem}
\begin{proof}
	We take the case that $B$ satisfies \eqref{eq:period1}.
	Equations \eqref{eq:riccatiW} and \eqref{eq:period1} imply $\tilde s$ is $N_0$-periodic.
	Like with the argument in the proof of Lemma \ref{lemma:unitary}, Condition $\mathcal{C}$ implies Equation \eqref{eq:period2} also holds, so $\hat s$ is also $N_0$-periodic. 
	Consequently, $\tilde{\hat s}=\hat{\tilde s}$ is $N_0$-periodic by Equation \eqref{eq:permutability}. Equation \eqref{eq:close} together with the Sym formula gives the result.
	The case that $D$ satisfies \eqref{eq:period2} is similar.
\end{proof}

\begin{figure}[htbp] 
  \setlength{\tabcolsep}{0pt}
  \renewcommand{\arraystretch}{0}
  \begin{tabular}{ccc}
    \includegraphics[width=5.0cm, trim=4.5cm 1cm 4.5cm 1cm, clip]{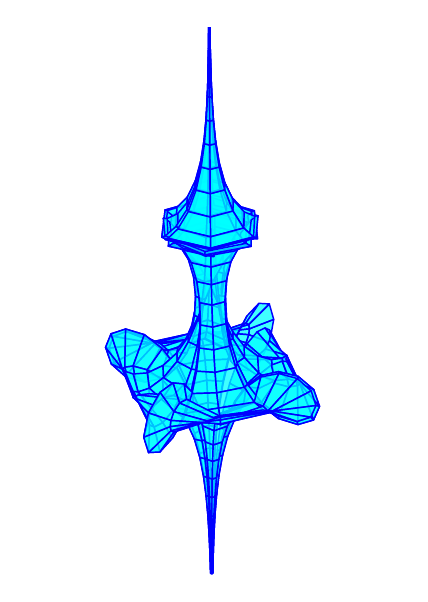} &
    \includegraphics[width=5.0cm, trim=4.5cm 1cm 4.5cm 1cm, clip]{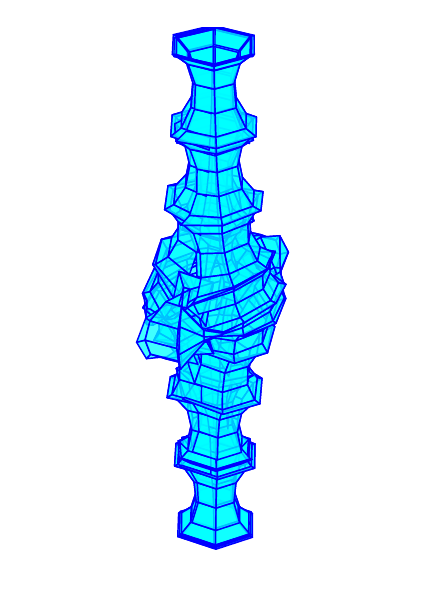} &  
    \includegraphics[width=5.0cm, trim=4.5cm 1cm 4.5cm 1cm, clip]{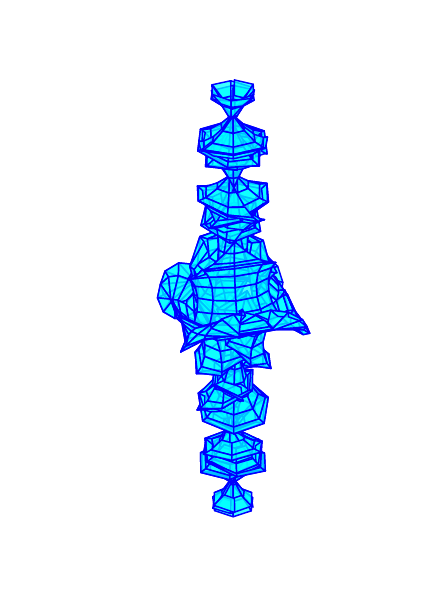}      \\
    \includegraphics[width=5.0cm, trim=4.5cm 1cm 4.5cm 1cm, clip]{DoubleBacklundpseudosphereperiod24paper.pdf} &
    \includegraphics[width=5.0cm, trim=4.5cm 1cm 4.5cm 1cm, clip]{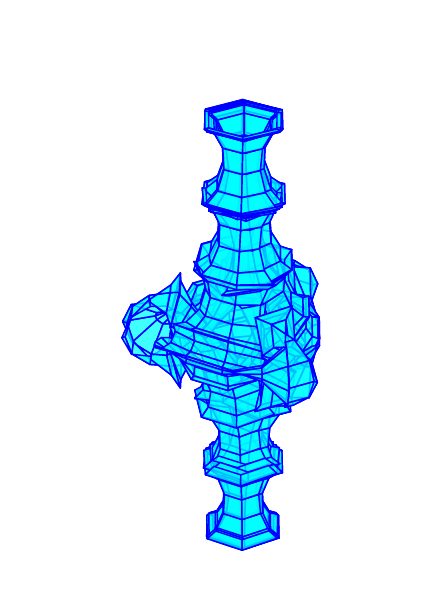} &    
    \includegraphics[width=5.0cm, trim=4.5cm 1cm 4.5cm 1cm, clip]{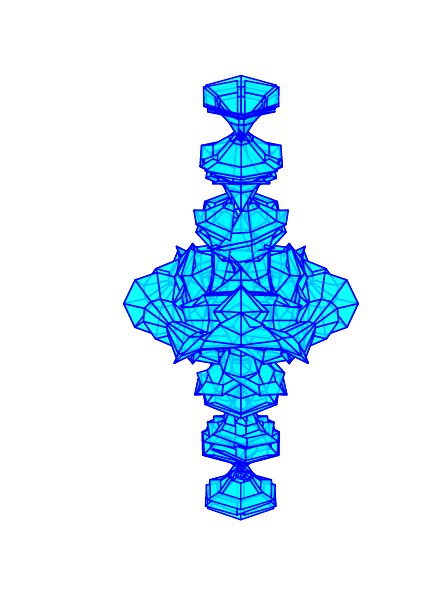}      
  \end{tabular}
  \caption{Upper row (left–right): the double B\"acklund transforms for rc-nets for $\kappa=1$, $0<\kappa<1$, $\kappa>1$, which are periodic with $k_0=6$, $N_0=9$.
	Lower row (left–right): the double B\"acklund transforms for rc-nets for $\kappa=1$, $0<\kappa<1$, $\kappa>1$, periodic with $k_0=6$, $N_0=8$.}
  \label{figure:backlund}
\end{figure}

\subsection{Linearization of the rational difference equations for $\tilde s$ and $\hat s$}
Finally, we describe linearization of the rational difference equations \eqref{eq:riccatiWW} and \eqref{eq:riccatiVV} 
when $B$, $D$ are $k$-invariant. 
As the two cases are similar, we only discuss Equation \eqref{eq:riccatiWW}, using a well known method for linearization of  
rational difference equations given one particular solution $\zeta$. Then 
$$\tilde S=\frac{1}{\tilde s-\zeta}$$
satisfies 
\begin{subequations} \label{eq:linear}
\begin{equation} \label{eq:linear1}
\tilde S(j+1,k)=\frac{(A_{21}\zeta(j,k)+A_{22})^2}{\det A_{(j+1,k)(j,k)}}\tilde S(j,k) +\frac{A_{21}(A_{21}\zeta(j,k)+A_{22})}{\det A_{(j+1,k)(j,k)}},
\end{equation}
\begin{equation} \label{eq:linear2}
\tilde S(j,k+1)=\frac{(B_{21}\zeta(j,k)+B_{22})^2}{\det B_{(j,k+1)(j,k)}}\tilde S(j,k)+\frac{B_{21}(B_{21}\zeta(j,k)+B_{22})}{\det B_{(j,k+1)(j,k)}}
\end{equation}
\end{subequations}
where $A_{(j+1,k)(j,k)}=(A_{mn})_{m,n=1,2}$, $B_{(j,k+1)(j,k)}=(B_{mn})_{m,n=1,2}$.
When $B$ is $k$-invariant, we can obtain the fixed points of \eqref{eq:B} in the $k$-direction
for each $j$, which can be regarded as particular solutions $\zeta$ that can appear as in the following lemma.

\begin{lemma} \label{lemma:zeta}
For the rational difference equations \eqref{eq:riccatiWW} with $k$-invariant $A$ and $B$, 
if $\vector{\zeta(j)}{1}$ is a $k$-invariant eigenvector of $B_{(j,k+1)(j,k)}$ for some $j$, then 
$\vector{\zeta(j+1)}{1}$ is a $k$-invariant eigenvector of $B_{(j+1,k+1)(j+1,k)}$, where 
$\zeta(j+1)=A_{(j+1,k)(j,k)}\bullet \zeta(j)$.
\end{lemma}
\begin{proof}
	By the compatibility condition and $k$-invariance, we have 
	$$B_{(j+1,k+1)(j+1,k)}A_{(j+1,k)(j,k)}=A_{(j+1,k)(j,k)}B_{(j,k+1)(j,k)}.$$
	Inserting $\vector{\zeta(j)}{1}$ into this, we have 
	$B_{(j+1,k+1)(j+1,k)}\vector{\zeta(j+1)}{1}\parallel \vector{\zeta(j+1)}{1}$.
	Because $A$ is $k$-invariant, $\zeta(j+1)$ is as well.
\end{proof}

We now arrive at the linearization:

\begin{theorem}
For the rational difference equations \eqref{eq:riccatiWW} with $k$-invariant $A$ and $B$ and 
$k$-invariant solution $\zeta(j)$, then
$\tilde s= \frac{1}{\tilde S}+\zeta$ with $\tilde S$ as in the linear equation \eqref{eq:linear} is another solution of \eqref{eq:riccatiWW}.
The analogous result holds for Equations \eqref{eq:riccatiVV}.
\end{theorem}

\bibliographystyle{amsplain}
\bibliography{reference}
\vspace{0.5cm}
\noindent Thomas Raujouan\\
Faculty of Science and Technology, University of Tours\\
raujouan@univ-tours.fr
\\

\noindent Wayne Rossman\\
Department of Mathematics, Graduate School of Science, Kobe University\\
wayne@math.kobe-u.ac.jp
\\

\noindent Naoya Suda\\
Department of Mathematics, Graduate School of Science, Kobe University\\
230s011s@stu.kobe-u.ac.jp

\end{document}